\documentclass[review,onefignum,onetabnum]{siamart190516}

\usepackage{mathtools,float,diagbox}
\usepackage{lipsum}
\usepackage{amsfonts,physics}
\usepackage{graphicx,bm}
\usepackage{epstopdf}

\newcommand{\Og}{\Omega}

\newcommand{\Dt}{\Delta}
\newcommand{\be}{\begin{equation}}
\newcommand{\ee}{\end{equation}}
\newcommand{\ba}{\begin{array}}
\newcommand{\ea}{\end{array}}
\newcommand{\bea}{\begin{eqnarray}}
\newcommand{\eea}{\end{eqnarray}}
\newcommand{\beas}{\begin{eqnarray*}}
\newcommand{\eeas}{\end{eqnarray*}}

\usepackage{graphics,color,cancel}
\usepackage{amsmath,amscd,amssymb,bm}
\usepackage{mathrsfs,cite}
\usepackage{fancyhdr,fancybox}
\usepackage{listings}
\usepackage{url} 
\usepackage{enumerate,epsfig}

\usepackage{array}
\usepackage{multirow}
\usepackage{subfig}
\ifpdf
  \DeclareGraphicsExtensions{.eps,.pdf,.png,.jpg}
\else
  \DeclareGraphicsExtensions{.eps}
\fi

\newsiamremark{remark}{Remark}
\newsiamremark{hypothesis}{Hypothesis}
\crefname{hypothesis}{Hypothesis}{Hypotheses}
\newsiamthm{claim}{Claim}

\headers{The Ramshaw-Mesina Hybrid Algorithm applied to the NSE}{A. \c{C}ibik, F. Siddiqua, and W. Layton}

\title{The Ramshaw-Mesina Hybrid Algorithm applied to the Navier Stokes Equations\thanks{
The research was partially supported by NSF grant DMS-2110379 and TUBITAK grant BIDEB2219.}}
\author{Aytekin \c{C}ibik\thanks{Department of Mathematics, Gazi University, Ankara, 06550, T\"{u}rkiye 
  (abayram@gazi.edu.tr ).}
  \and  Farjana Siddiqua\thanks{Department of Mathematics, University of Pittsburgh, Pittsburgh, PA-15260
  (fas41@pitt.edu ).}
\and William Layton\thanks{Department of Mathematics, University of Pittsburgh, Pittsburgh, PA-15260 
  (wjl@pitt.edu).}
}

\usepackage{amsopn}

\ifpdf
\hypersetup{
  pdftitle={The Ramshaw-Mesina Hybrid Algorithm applied to the Navier Stokes Equations },
  pdfauthor={Aytekin \c{C}{\i}b{\i}k, Farjana Siddiqua, and William Layton}
}
\fi

\begin{document}
\nolinenumbers
\maketitle
\begin{abstract}
In 1991, Ramshaw and Mesina proposed a novel synthesis of penalty methods and artificial compression methods. When the two were balanced they found the combination was 3-4 orders more accurate than either alone. This report begins the study of their interesting method applied to the Navier-Stokes equations. We perform stability analysis, semi-discrete error analysis, and tests of the algorithm. Although most of the results for implicit time discretizations of our numerical tests comply with theirs for explicit time discretizations, the behavior in damping pressure oscillations and violations of incompressibility are different from their findings and our heuristic analysis.
\end{abstract}
\section{Introduction}
In 1991, Ramshaw and Mesina proposed and tested a numerical regularization for approximating solutions of the incompressible Navier-Stokes equations by combining pressure penalty (PP) and artificial compression (AC) methods in a novel way \cite{ramshaw1991hybrid}. The numerical results of Ramshaw and Mesina indicated that their method is more accurate compared to PP and AC methods in terms of computational time and divergence errors,
\begin{center}
... \emph{3-4 orders of magnitude smaller}..., \cite[page 170, point 3]{ramshaw1991hybrid}.
\end{center}
In this study, we investigate the hybrid scheme of Ramshaw and Mesina with finite element spatial discretization and implicit time discretization. The unconditional stability and convergence results are verified by our numerical tests.
Our tests did not show the above advantage in damping pressure oscillations and violations of incompressibility, possibly due to differences between implicit and explicit time discretizations.\par
Consider a regular and bounded flow domain $\Omega\subset  \mathbb{R}^d \ (d=2,\ 3)$. The Navier-Stokes equations (NSE) with no-slip boundary conditions are:
\begin{equation}\label{nse0}
\begin{aligned}
& u_t+ u\cdot\grad u-\nu\Delta u+\grad p=f(x),\ \text{and} \ \div u=0,\ \text{in}\ \Omega\cross (0,T],
\\& u=0,\ \text{on}\ \partial \Omega \cross(0,T],\ \text{and}\ u(x,0)=u_0,\ \text{in}\ \Omega.
\end{aligned}
\end{equation}
Here $u$ is the velocity, $\nu$ is the kinematic viscosity, $p$ is the pressure, and $f$ is the prescribed body force. One of the major challenges for solving the NSE numerically is the coupling of velocity and pressure which increases the execution times of codes and raises the computational memory needs. PP and AC methods are known to perform well in terms of eliminating this coupling \cite{guzel22}.
The idea of Ramshaw and Mesina \cite{ramshaw1991hybrid} is to convert the penalty relaxation of incompressibility from $2\beta\div u+p=0$ to $\frac{d}{dt}(2\beta \div u+p)=0$ before combining with the  artificial compression relaxation $p_t+\alpha^2\div u=0$. This gives rise to a hybrid model which is then discretized in space and time. Their motivating intuition was that the penalty method damps rapidly high frequency components of incompressibility violations while the artificial compression method reduces those violations by moving them to higher frequencies as non-physical acoustics. This decouples the velocity and pressure, making the resulting system significantly more straightforward to solve. They considered an explicit finite difference method for discretization in space and time.\par
The algorithm we consider to approximate the NSE \eqref{nse0} solution $(u,p)$ is the FEM discretization of the following continuum model. Skew-symmetrize the non-linearity by adding $\frac{1}{2}(\div w)w$, select large parameters $\alpha^{2}$\ and $\beta$:
\begin{equation}\label{nse}
\begin{aligned}
& w_t+ w\cdot\grad w+\frac{1}{2}(\div w)w-\nu\Delta w+\grad \lambda=f(x),\ \text{in}\ \Omega\cross (0,T],
\\& \lambda_t+2\beta \div w_t+\alpha^2\div w=0,\ \text{in}\ \Omega\cross (0,T],
\\& w=0,\ \text{on}\ \partial \Omega \cross(0,T],\ \text{and}\ w(x,0)=w_0,\ \text{in}\ \Omega.
\end{aligned}
\end{equation}

\subsection{Related Work}
Ramshaw and Mesina \cite{ramshaw1991hybrid} found the hybrid combination very effective in damping $\|\div u\|$ in explicit time discretizations. It also has the property that at steady state $(p_t=0, u_t =0)$ exact incompressibility holds. Thus, several (Ramshaw and Mousseau \cite{ramshaw1990accelerated,ramshaw1991damped}, McHugh and Ramshaw \cite{mchugh1995damped}, \c{C}{\i}b{\i}k and Layton \cite{ccibik2024convergence}) studied it as a dynamic iteration to solve the steady state flow equations. Formerly, Kobel'kov \cite{kobel2002symmetric} and Brooks and Hughes (in a short remark in a long paper) \cite{brooks1982streamline} have tried syntheses of penalty and artificial compression models without significant gain. Dukowicz \cite{duk93} found fully explicit treatment of both velocity and pressure was less efficient that fully implicit treatment with CG solvers. Generally, there is a close connection between the analysis of artificial compression methods and penalty methods, developed in the book of Prohl \cite{prohl1997projection}. There are a huge number of papers on both approaches. We note some early works in \cite{heinrich1995penalty,shen1995error}. Interesting developments continue in e.g. \cite{olshanskii2019penalty,linke2017connection,layton2020doubly,fang2023penalty}, among many recent papers.\par
To the best of our knowledge, this is the first study concerning the hybrid scheme of Ramshaw and Mesina in terms of the finite element method. The Ramshaw-Mesina  idea allows an implicit velocity update with an explicit pressure update (preserving unconditional stability, \Cref{thm:be}, \Cref{sec:4}), an option not investigated previously. 
We also consider the damping effect of the hybrid scheme and test and compare its performance with PP and AC methods. \par
 The paper is organized as follows: The algorithm is described in \Cref{sec:3}, the stability analysis and the error analysis are given in \Cref{sec:4} and \Cref{sec:5}, respectively. Numerical tests are presented in \Cref{sec:6}.

\section{Notation and Preliminaries}
In this section, we introduce some of the notations and results used in this paper. We denote by $\|\cdot\|$ and $(\cdot,\cdot)$ the $L^2(\Omega)$ norm and inner product, respectively. We denote the $L^p(\Omega)$ norm by $\|\cdot\|_{L^p}$.
The solution spaces $X$ for the velocity and $Q$ for the pressure are defined as:
\begin{equation*}
\begin{aligned}
    &X:=(H_0^1(\Omega))^d=\{ v\in (L^2(\Omega))^d: \grad  v\in (L^2(\Omega))^{d\cross d}\ \text{and}\  v=0\ \text{on}\ \partial\Omega\},\\
    &Q:=L^2_0(\Omega)=\{q\in L^2(\Omega): \int_\Omega q\ d x=0\}.
\end{aligned}
\end{equation*}
We denote Bochner Space \cite{adams2003sobolev} norm by $\|v\|_{L^p(0,T;X)}=\Bigg(\int_0^T\|v(\cdot,t)\|_{X}^p dt\Bigg)^{\frac{1}{p}}$, $p\in [1,\infty)$.
The space $H^{-1}(\Omega)$ denotes the dual space of bounded linear functionals defined on $H_0^1(\Omega)=\{ v\in H^1(\Omega):  v=0\ \text{on}\  \partial\Omega\}$ and this space is equipped with the norm:
\begin{equation*}
    \|f\|_{-1}=\sup_{0\neq  v\in X}\frac{(f, v)}{\|\grad  v\|}.
\end{equation*}

The finite element method for this problem involves picking finite element spaces \cite{cfdbook} $X^h\subset X$ and $Q^h\subset Q$. We assume that $(X^h, Q^h)$ are conforming and satisfy the following approximation properties \cref{prop} and discrete inf-sup condition \cref{infsup}: For $u\in (H^{m+1}(\Omega))^d\ \cap\ H_0^1(\Omega)$ and $p\in H^{m}(\Omega)$,
\begin{equation}\label{prop}
\begin{aligned}
\inf_{v^h\in X^h}\{\|u-v^h\|+h\|\grad(u-v^h)\|\}&\leq Ch^{m+1}|u|_{m+1},
\\\inf_{q^h\in Q^h}\|p-q^h\|&\leq Ch^{m}|p|_{m},
\end{aligned}
\end{equation}
\begin{equation}\label{infsup}
\inf_{q^h\in Q^h}\sup_{ v^h\in X^h}\frac{(q^h,\grad\cdot  v^h)}{\|q^h\|\|\grad  v^h\|}\geq \gamma^h>0,
\end{equation}
where $\gamma^h$ is bounded away from zero uniformly in $h$.
\begin{definition}\label{def:l2proj}
The projection operators ${\Pi}_{X}: L^2(\Omega)\rightarrow X^h$ and ${\Pi}_{Q}: L_0^2(\Omega)\rightarrow Q^h$, satisfy
  \begin{align}
      (u-{\Pi}_{X}(u),v^h)&=0,\ \forall v^h\in X^h,\label{l2proju}
      \\ (r-{\Pi}_{Q}(r),q^h)&=0,\ \forall q^h\in Q^h.\label{l2proj}
  \end{align}
The Stokes projection \cite{labovsky2009stabilized} operator $\Pi_S: (X,Q)\rightarrow(X^h,Q^h)$, $\Pi_S(u,p)=(\widetilde{u},\widetilde{p})$, satisfies
\begin{equation}\label{stokes1}
  \begin{aligned}
      \nu(\grad(u-\widetilde{u}),\grad v^h)-(p-\widetilde{p},\div v^h)&=0,\ \forall v^h\in X^h,
      \\ (\div(u-\widetilde{u}),q^h)&=0,\ \forall q^h\in Q^h.
  \end{aligned}
  \end{equation}
\end{definition}
\par
The following estimate in $H^{-1}(\Omega)$, provable for standard finite element spaces under mild assumptions \cite[p.160]{cfdbook}, is assumed,
\begin{equation}\label{proj}
   \|u-{\Pi}_{X}(u)\|_{-1} \leq C h\|u-{\Pi}_{X}(u)\|.
\end{equation}
\begin{definition}\label{trilinear0}
We define the skew symmetrized trilinear form $b^*:X\times{X}\times{X}\rightarrow \mathbb{R}$ as follows: 
\begin{align*}
b^*( u, v, w):=(u\cdot\grad v,w)+\frac{1}{2}((\div u)v,w).
\end{align*}
\end{definition}
\begin{lemma}\label{trilinear_ineq}
(see \cite[p.11, p.123, p.155]{cfdbook}). There exists $C_1$ and $C_2$ such that for any $ u,\  v,\  w\in X$, the skew-symmetric tri-linear form, $b^{*}(u,v,w)$ satisfies
\begin{equation*}
\begin{aligned}
    & \left|b^{*}(u,v,w)\right|\leq C_1(\Omega)\|\grad  u\|\|\grad  v\|\|\grad  w\|,\\&
    \left|b^{*}(u,v,w)\right|\leq C_2(\Omega)\| u\|^{1/2}\|\grad  u\|^{1/2}\|\grad  v\|\|\grad  w\|.
    \end{aligned}
\end{equation*}
\end{lemma}
\begin{proposition}\label{errstokes}(Error Estimation for the Stokes Projection) Suppose the discrete inf-sup condition \eqref{infsup} holds. Let $C_3$ be a constant independent of $h$ and $\nu$, and $C_4=C(\nu,\Omega)$. 
The error in the Stokes Projection \eqref{stokes1} satisfies
\begin{align}
    \| p-\widetilde{p}\|&\leq \frac{\nu}{\gamma^h}\|\grad(u-\widetilde{u})\| \label{eq1},
    \\
    \nu\|\grad(u-\widetilde{u})\|^2&\leq C_3\bigg[\nu\inf_{v^h\in X^h}\|\grad (u-v^h)\|^2+\nu^{-1} \inf_{q^h\in Q^h}\| p-q^h\|^2\bigg]\label{eq2}.
    \end{align}
    Furthermore, let $\Omega$ be such that the Stokes problem is $H^2$ regular. In that case, the $L^2$ error in the Stokes Projection \eqref{stokes1} satisfies
    \begin{align}
     \|u-\widetilde{u}\|&\leq C_4 h\bigg(\inf_{v^h\in X^h}\|\grad (u-v^h)\|+\inf_{q^h\in Q^h}\| p-q^h\|\bigg)\label{eq3}.
    \end{align}
\end{proposition}
\begin{proof}
For the proof of \eqref{eq1}, see \cite[Proposition $4.2$]{kean2023doubly}. Inequality \eqref{eq2} is proved in \cite[Proposition 2.2]{labovsky2009stabilized}. Using the Aubin-Nitsche lift, one can obtain \eqref{eq3}, see \cite[p.373]{baker1982higher}.
\end{proof}
\begin{remark}
To have an estimator for $\|\grad(u-\widetilde{u})_t\|$, we take the partial derivative of \eqref{stokes1} with respect to time $t$,
\begin{equation}\label{stokest}
  \begin{aligned}
      \nu(\grad(u-\widetilde{u})_t,\grad v^h)-((p-\widetilde{p})_t,\div v^h)&=0,\ \forall v^h\in X^h,
      \\ (\div(u-\widetilde{u})_t,q^h)&=0,\ \forall q^h\in Q^h.
  \end{aligned}
  \end{equation}
If the finite element space does not evolve with time (e.g. by mesh movement), then $(\widetilde{u})_t=\widetilde{(v_t)}$ and $(\widetilde{p})_t=\widetilde{(p_t)}$ and we have,
\begin{equation}\label{errst1}
\begin{aligned}
   & \| (p-\widetilde{p})_t\|\leq \frac{\nu}{\gamma^h}\|\grad(u-\widetilde{u})_t\|,
    \\&
    \nu\|\grad(u-\widetilde{u})_t\|^2\leq C\bigg[\nu\inf_{v^h\in X^h}\|\grad (u-v^h)_t\|^2+\nu^{-1} \inf_{q^h\in Q^h}\| (p-q^h)_t\|^2\bigg].
    \end{aligned}
\end{equation}
We assume \eqref{errst1} holds in our error analysis of \Cref{sec:5}.
\end{remark}

\section{Formulation}\label{sec:3}
We can rewrite the weak formulation of NSE \eqref{nse0} as follows: Find $( u,p)\in (X,Q)$ such that
\begin{align}
\label{nse_re_v1}
    ( u_t, v)  +b^{*}(u,u,v)+\nu(\grad u,\grad  v)
    -( p,\grad\cdot  v)&=(f, v),\quad\forall\   v\in X,\\
    \label{nse_re_v2}
    (p_t+2\beta \div u_t+\alpha^{2}\div u,q)&=(p_t,q), \quad\forall\ q\in Q.
\end{align}
We write the weak formulation of (\ref{nse})
as follows:
Find $( w,\lambda)\in (X,Q)$ such that
\begin{align}
\label{nse_v1}
    ( w_t, v)  +b^{*}(w,w,v)+\nu(\grad w,\grad  v)
    -( \lambda,\grad\cdot  v)&=(f, v),\quad\forall\   v\in X,\\
    \label{nse_v2}
    (\lambda_t+2\beta \div w_t+\alpha^{2}\div w,q)&=0, \quad\forall\ q\in Q.
\end{align}
Next, we consider the semi-discrete approximation of (\ref{nse}).
Suppose $ w^h( x,0)$ is approximation of $ w( x,0)$. The approximate velocity and pressure are maps
\begin{equation*}
     w^h: [0,T]\to X^h, \ \lambda^h:(0,T]\to Q^h 
\end{equation*}
\text{satisfying}, $\forall\   v^h\in X^h \ \text{and}\ \forall\  q^h\in Q^h$,
\begin{align}
\label{nse_semi_v1}
    ( w_t^h, v^h)  +b^{*}(w^h,w^h,v^h)+\nu(\grad w^h,\grad  v^h)
    -( \lambda^h,\grad\cdot  v^h)&=(f, v^h),\\
    \label{nse_semi_v2}
    (\lambda_t^h+2\beta \div w_t^h+\alpha^{2}\div w^h,q^h)&=0.
\end{align}
Let the time step and other quantities be denoted by 
\begin{align*}
    \text{time-step}=k,\ t_n=nk,\ f_n( x)=f( x,t_n), \\
     w_n^h( x)=\text{approximation to}\  u( x,t_n), \\
    \lambda_n^h( x)=\text{approximation to}\ p( x,t_n).
\end{align*}
To get a full discretization, we consider finite element spatial discretization and the first-order Backward Euler scheme for time discretization. We also test a higher order scheme in \Cref{sec:CN}.
Given $( w_n^h,\lambda_n^h)\in (X^h,Q^h)$, for all $v^h\in X^h$ and $q^h\in Q^h$, find $( w_{n+1}^h,\lambda_{n+1}^h)\in (X^h,Q^h)$ satisfying
\begin{equation}\label{nse_fem_be1}
\begin{aligned}
  &\Big( \frac{ w_{n+1}^h- w_n^h}{k}, v^h\Big)+b^*( w^h_n, w^h_{n+1}, v^h) 
    +\nu(\grad  w^h_{n+1},\grad  v^h)-(\lambda^h_{n+1},\div  v^h)
    \\&=(f_{n+1}( x), v^h),
    \end{aligned}
\end{equation}    
\begin{equation}
    \begin{aligned}
    \label{nse_fem_be2}
    &\Big( \frac{ \lambda_{n+1}^h- \lambda_n^h}{k}, q^h\Big)+2\beta\Big( \div(\frac{  w_{n+1}^h- w_n^h}{k}), q^h\Big)
    +\alpha^{2}(\div w_{n+1}^h,q^h)=0.
\end{aligned}
\end{equation}
This method is semi-implicit. 
Equation \eqref{nse_fem_be2} says that
$$\lambda_{n+1}^h=\Pi_{Q}(\lambda_n^h-(k\alpha^2+2\beta)\div w_{n+1}^h+2\beta\div w_n^h).$$
Notice that $\lambda_{n+1}^h\in Q^h$ is well defined since $\int_{\Omega}\div w^h_{n+1}d\Omega=\int_{\partial \Omega} w_{n+1}^h\cdot n \ ds=0$. 
Inserting this in \eqref{nse_fem_be1} gives
\begin{equation}\label{nse_fem_be12}
\begin{aligned}
   &( \frac{ w_{n+1}^h- w_n^h}{k}, v^h)+b^*( w^h_n, w^h_{n+1}, v^h)
    +\nu(\grad  w^h_{n+1},\grad  v^h)
    \\& -2\beta( \Pi_{Q}(\div w^h_{n}),\div v^h)
    + (k\alpha^2+2\beta) (\Pi_{Q}( \div w^h_{n+1}), \div v^h)
    \\&=(f_{n+1}( x), v^h) + (\lambda^h_{n},\div v^h),
    \end{aligned}
\end{equation}    
\begin{theorem}
The fully coupled method \eqref{nse_fem_be1} and \eqref{nse_fem_be2} is equivalent to \eqref{nse_fem_be12} and \eqref{nse_fem_be2}. 
\end{theorem}
\begin{remark}
We can also decouple before space discretization then discretize. This yields almost the same system except the projection $\Pi_{Q}$ do not occur. This gives the alternate form used in the numerical tests: find $( w_{n+1}^h,\lambda_{n+1}^h)\in (X^h,Q^h)$ satisfying
\begin{equation}\label{nse_fem_be13}
\begin{aligned}
   &\Big( \frac{ w_{n+1}^h- w_n^h}{k}, v^h\Big)+b^*( w^h_n, w^h_{n+1}, v^h)
    +\nu(\grad  w^h_{n+1},\grad  v^h)
    \\& -2\beta( \div w^h_{n},\div v^h)
    + (k\alpha^2+2\beta) ( \div w^h_{n+1}, \div v^h)
   \\& =(f_{n+1}( x), v^h) + (\lambda^h_{n},\div v^h),
    \end{aligned}
\end{equation}    
\begin{equation}
    \begin{aligned}
    \label{nse_fem_be23}
    &( \lambda_{n+1}^h, q^h)=( \lambda_{n}^h, q^h)-(k\alpha^2+2\beta) ( \div w^h_{n+1},q^h )+2\beta (\div w^h_{n},q^h).
\end{aligned}
\end{equation}
\end{remark}

%
\section{Stability}\label{sec:4}
In this section, we prove the unconditional stability of the model \eqref{nse} and the semi-implicit method (\eqref{nse_fem_be1} \& \eqref{nse_fem_be2}) in \cref{thm:stability} and \Cref{thm:be}, respectively.
\begin{theorem}\label{thm:stability}
(Stability of $ w$) \eqref{nse} is unconditionally stable. The solution $ w$ of \eqref{nse}  satisfies the following inequality
\begin{align*}
     &\| w(T)\|^2+\alpha^{-2}\| (\lambda+2\beta\div w)(T)\|^2
    +\int_0^T\Big(\nu\|\grad w\|^2+4\beta \|\div w\|^2\Big)\ dt
    \\&
    \leq\int_0^T \frac{1}{\nu}\|f\|_{-1}^2\, dt+\| w_0\|^2+\alpha^{-2}\| \lambda_0+2\beta\div w_0\|^2 .
\end{align*}
\end{theorem}
\begin{proof}
 Take $v=w$ in \eqref{nse_v1}, $q=\lambda+2\beta\div w$ in \eqref{nse_v2}, and add them. We get the following energy equality,
\begin{equation}\label{energy_equality}
\begin{aligned}
    &\frac{1}{2}\frac{d}{dt}\| w\|^2+\nu\|\grad w\|^2-(\lambda,\div w)+(\div w,\lambda+2\beta\div w)
    \\&+\frac{\alpha^{-2}}{2}\frac{d}{dt}\| \lambda+2\beta\div w\|^2=(f, w).
    \end{aligned}
\end{equation}
Notice that $-(\lambda,\div w)+(\div w,\lambda+2\beta\div w)=(\div w,2\beta \div w)=2\beta \|\div w\|^2$. Hence, we can rewrite \eqref{energy_equality} as,
\begin{equation}\label{energy_equality2}
    \frac{1}{2}\frac{d}{dt}\| w\|^2+\nu\|\grad w\|^2+2\beta \|\div w\|^2+\frac{\alpha^{-2}}{2}\frac{d}{dt}\| \lambda+2\beta\div w\|^2=(f, w).
\end{equation}
Using Young's inequality at the right hand side of \eqref{energy_equality2} we get,
\begin{equation}\label{energy_inequality}
   \frac{d}{dt}(\| w\|^2+\alpha^{-2}\| \lambda+2\beta\div w\|^2)+\nu\|\grad w\|^2+4\beta \|\div w\|^2\leq \frac{1}{\nu}\|f\|_{-1}^2.
\end{equation}
Integrating \eqref{energy_inequality} from $t=0$ to $t=T$, we have
\begin{align*}
    &\Big(\| w(T)\|^2+\alpha^{-2}\| (\lambda+2\beta\div w)(T)\|^2\Big)
    +\int_0^T\Big(\nu\|\grad w\|^2+4\beta \|\div w\|^2\Big)\ dt
    \\&
    \leq\int_0^T \frac{1}{\nu}\|f\|_{-1}^2\, dt+\Big(\| w(0)\|^2+\alpha^{-2}\| (\lambda+2\beta\div w)(0)\|^2\Big) .
\end{align*}
\end{proof}
Recall that ${\Pi}_{Q}$ is the $L_0^2$ projection (\Cref{def:l2proj}) into $Q^h$.
\begin{remark}
(Stability of $ w^h$) \eqref{nse} is unconditionally stable. The semi-discrete solution $ w^h$ of \eqref{nse}  satisfies the following inequality
\begin{align*}
     &\| w^h(T)\|^2+\alpha^{-2}\| (\lambda^h+2\beta\Pi_{Q}(\div w^h))(T)\|^2
   \\& +\int_0^T\Big(\nu\|\grad w^h\|^2+4\beta \|\Pi_{Q}(\div w^h)\|^2\Big)\ dt
 \\&
\leq\int_0^T \frac{1}{\nu}\|f\|_{-1}^2\, dt+\| w^h(0)\|^2+\alpha^{-2}\| (\lambda^h+2\beta\Pi_{Q}(\div w^h))(0)\|^2 .
\end{align*}
\end{remark}
\begin{theorem}\label{thm:be}
The method  \eqref{nse_fem_be1} and \eqref{nse_fem_be2} is unconditionally energy stable. For any $N\geq 1$,
\begin{gather*}
     \frac{1}{2}\| w^h_N\|^2+\frac{\alpha^{-2}}{2}\|  \lambda^h_{N}+2\beta{\Pi}_{Q}(\div w_{N}^h)\|^2
    +k\sum_{n=0}^{N-1} (
    2\beta\|{\Pi}_{Q}(\div w_{n+1}^h)\|^2
    \\+\nu\|\grad  w^h_{n+1}\|^2)
    +\sum_{n=0}^{N-1}\frac{1}{2}\Big(
    \alpha^{-2}\|\lambda^h_{n+1}+2\beta{\Pi}_{Q}(\div w_{n+1}^h)
    -\lambda^h_{n}-2\beta{\Pi}_{Q}(\div w_{n}^h)\|^2
    \\+\| w^h_{n+1}- w^h_n\|^2\Big) 
    =\frac{1}{2}\| w^h_0\|^2+\frac{\alpha^{-2}}{2}\|  \lambda^h_{0}+2\beta{\Pi}_{Q}(\div w_{0}^h)\|^2+k\sum_{n=0}^{N-1} (f_{n+1}, w^h_{n+1}).
\end{gather*}
\end{theorem}
\begin{proof}
 Take  $ v^h= w^h_{n+1}$ in \eqref{nse_fem_be1} and $q^h=\lambda^h_{n+1}+2\beta{\Pi}_{Q}(\div w_{n+1}^h)$ in \eqref{nse_fem_be2}, and add them. Note that $b^*( w^h_n, w^h_{n+1}, w^h_{n+1})$ $=0$. Hence after multiplying by $k$, we get,
\begin{gather*}
    \| w^h_{n+1}\|^2-( w^h_{n+1}, w^h_n)+\alpha^{-2}\|  \lambda^h_{n+1}+2\beta{\Pi}_{Q}(\div w_{n+1}^h)\|^2
    \\-\alpha^{-2}( \lambda^h_{n+1}+2\beta{\Pi}_{Q}(\div w_{n+1}^h), \lambda^h_{n}
    +2\beta{\Pi}_{Q}(\div w_{n}^h)) 
    -k(\lambda_{n+1}^h,\div w_{n+1}^h)
    \\+k(\div w_{n+1}^h, \lambda^h_{n+1}+2\beta{\Pi}_{Q}(\div w_{n+1}^h))+k \nu\|\grad  w^h_{n+1}\|^2= k(f_{n+1}, w^h_{n+1}).
\end{gather*}
For the second and fourth term, apply the polarization identity, 
\begin{gather*}
    ( w^h_{n+1}, w^h_n)=\frac{1}{2}\| w^h_{n+1}\|^2+\frac{1}{2}\| w^h_n\|^2-\frac{1}{2}\| w^h_{n+1}- w^h_n\|^2, \\
    \alpha^{-2}( \lambda^h_{n+1}+2\beta{\Pi}_{Q}(\div w_{n+1}^h), \lambda^h_{n}+2\beta{\Pi}_{Q}(\div w_{n}^h))
    \\=\frac{\alpha^{-2}}{2}\| \lambda^h_{n+1}+2\beta{\Pi}_{Q}(\div w_{n+1}^h)\|^2
    +\frac{\alpha^{-2}}{2}\|  \lambda^h_{n}+2\beta{\Pi}_{Q}(\div w_{n}^h)\|^2
    \\-\frac{\alpha^{-2}}{2}\|\lambda^h_{n+1}+2\beta{\Pi}_{Q}(\div w_{n+1}^h)- \lambda^h_{n}-2\beta{\Pi}_{Q}(\div w_{n}^h)\|^2.
\end{gather*}
Next, we have,
\begin{align*}
&-k(\lambda_{n+1}^h,\div w_{n+1}^h)+k(\div w_{n+1}^h, \lambda^h_{n+1}+2\beta{\Pi}_{Q}(\div w_{n+1}^h))
\\&=k(\div w_{n+1}^h,2\beta{\Pi}_{Q}(\div w_{n+1}^h))
\\&=2k\beta\|{\Pi}_{Q}(\div w_{n+1}^h)\|^2.
\end{align*}
Collecting terms and summing from $n=0$ to $N-1$, we get the desired result.
\end{proof}

\section{Error Analysis}\label{sec:5}
In this section, we analyze the error between the strong solution to NSE and semi-discrete solutions to \eqref{nse} in \cref{thm:numerial-error1}. 
\begin{theorem}\label{thm:numerial-error1}
(Numerical error of semi-discrete case) Let $(X^h,Q^h)$ be the finite element spaces satisfying \cref{prop} and \cref{infsup}. Let $ u$ be a strong solution of the NSE \eqref{nse0}. Suppose the estimate \cref{proj} in $H^{-1}(\Omega)$ holds and $\grad u\in L^4(0,T;L^2(\Omega))$. Let
\begin{equation*}
a(t):=\max\{2,\frac{1}{2}+C(\nu)\|\grad u\|^4\}.
\end{equation*}
Then, we have the following error estimate:
\begin{gather*}
    \sup_{0\leq t\leq T}(\|(u-w^h)(t)\|^2+2\alpha^{-2}\|2\beta{\Pi}_{Q}(\div{(u-w^h))(t)}\|^2)
    \\+\int_0^T\Big( \frac{\nu}{4}\|\grad (u-w^h)\|^2+2\beta\|{\Pi}_{Q}(\div (u-w^h))\|^2\Big) dt \\
    \leq e^{\int_0^T a(t)dt}
    \Big\{\Big(\| (u-w^h)(0)\|^2+\alpha^{-2}\| (p-\lambda^h)(0)+2\beta{\Pi}_{Q}(\div (u-w^h))(0)\|^2\Big)
    \\+C(\nu,\Omega)\Big[h^{2m}
    (\|u\|^2_{L^4(0,T;H^{m+1}(\Omega))}
    +\|p\|^2_{L^4(0,T;H^{m}(\Omega))})\Big]+{\alpha^{-2}}\|p_t\|_{L^2(0,T;L^2(\Omega))}^2
    \\
    +\Big[(C(\nu,\Omega)(h^{2m}+
   h^{2+2m})(\|u_t\|^2_{L^2(0,T;H^{m+1}(\Omega))}
    +\|p_t\|^2_{L^2(0,T;H^{m}(\Omega))})\Big]
    \Big\}.
\end{gather*}
\end{theorem}
\begin{proof}
First, we combine \eqref{nse_re_v1} and \eqref{nse_re_v2} to get the following equation,
\begin{equation}\label{nse_v3}
\begin{aligned}
&( u_t, v)  +b^{*}(u,u,v)+\nu(\grad u,\grad  v)
    -( p,\div  v)
    \\&+\alpha^{-2}(p_t+2\beta\div u_t,q)+(\div u,q)
=\alpha^{-2}(p_t,q)+(f, v).
\end{aligned}
\end{equation}
Next, we combine \eqref{nse_semi_v1} and \eqref{nse_semi_v2} to get the following equation,
\begin{equation}\label{nse_semi_v3}
\begin{aligned}
&( w_t^h, v^h)  +b^{*}(w^h,w^h,v^h)+\nu(\grad w^h,\grad  v^h)
    -( \lambda^h,\div  v^h)
    \\&+\alpha^{-2}(\lambda_t^h+2\beta\div w_t^h,q^h)+(\div w^h,q^h)
    =(f, v^h).
\end{aligned}
\end{equation}
Since $ v\in X$\  \& \ 
$ v^h\in X^h\subset X$,
we restrict $ v= v^h$ in continuous variational problem. Similarly, we restrict $ q= q^h$ in continuous variational problem since $ q\in Q$ and $ q^h\in Q^h\subset Q$. Then, subtract semi-discrete problem \eqref{nse_semi_v3} from continuous problem \eqref{nse_v3}. Let ${e_u}= u- w^h$. This gives, 
\begin{equation}\label{err_1}
\begin{aligned}
  &\bigg(\frac{d}{dt}e_u,v^h\bigg)+b^{*}(u,u,v^h)-b^{*}(w^h,w^h,v^h)+\nu(\grad e_u,\grad  v^h)-( p- \lambda^h,\div  v^h)
  \\&
  +\alpha^{-2}\bigg(\frac{d}{dt}(p- \lambda^h+2\beta\div e_u),q^h\bigg)+(\div e_u,q^h)=\alpha^{-2}(p_t,q^h).
\end{aligned}
\end{equation} 
Let $\widetilde{ u}\in X^h$ and $\widetilde{ p}\in Q^h$ be the Stokes projection of $(u,p)$. 
Let $\eta_u= u- \widetilde{ u},\  \text{and}\
 \phi^h_u= w^h-\widetilde{ u}$.
This implies $e_u=\eta_u- \phi^h_u$. 
Then $\eqref{err_1}$ becomes
\begin{gather*}
 \bigg(\frac{d}{dt}\phi^h_u,v^h\bigg)+\nu(\grad \phi^h_u,\grad  v^h)-( \lambda^h-\widetilde{ p},\div  v^h)
  +\alpha^{-2}\bigg(\frac{d}{dt}(\lambda^h-\widetilde{ p}+2\beta\div \phi_u^h),q^h\bigg)
  \\+(\div \phi^h_u,q^h)
  = \bigg(\frac{d}{dt}\eta_u,v^h\bigg)+\nu(\grad \eta_u,\grad  v^h)-( p-\widetilde{p},\div  v^h)+\alpha^{-2}(p_t,q^h)
 \\ +\alpha^{-2}\bigg(\frac{d}{dt}(p-\widetilde{p}+2\beta\div \eta_u),q^h\bigg)
  +(\div \eta_u,q^h)+b^{*}(u,u,v^h)-b^{*}(w^h,w^h,v^h).
\end{gather*}
 Take $ v^h= \phi_u^h$ and $q^h=\lambda^h-\widetilde{ p}+2\beta {\Pi}_{Q}(\div \phi_u^h) $. Notice that $-( \lambda^h-\widetilde{ p},\div  \phi_u^h)+(\div \phi^h_u,\lambda^h-\widetilde{ p}+2\beta {\Pi}_{Q}(\div \phi_u^h))=2\beta {\Pi}_{Q}(\|\div  \phi_u^h\|^2)$. Furthermore due to the Stokes projection \eqref{stokes1}, $\nu(\grad \eta_u,\grad  \phi_u^h)-( p-\widetilde{p},\div  \phi_u^h)=0$ and $(\div \eta_u,\lambda^h-\widetilde{ p}+2\beta {\Pi}_{Q}(\div \phi_u^h))=0$.
Hence,
\begin{gather*}
    \frac{1}{2}\frac{d}{dt}\left\{\| \phi^h_u\|^2+\alpha^{-2}\| \lambda^h-\widetilde{ p}+2\beta {\Pi}_{Q}(\div \phi_u^h)\|^2\right\}+\nu\|\grad \phi^h_u\|^2 +2\beta \|{\Pi}_{Q}(\div  \phi_u^h)\|^2
    \\= \bigg(\frac{d}{dt}\eta_u,\phi_u^h\bigg)
  +\alpha^{-2}\bigg(\frac{d}{dt}(p-\widetilde{p}+2\beta\div \eta_u),\lambda^h-\widetilde{ p}+2\beta {\Pi}_{Q}(\div \phi_u^h)\bigg)
  \\+\alpha^{-2}(p_t,\lambda^h-\widetilde{ p}+2\beta {\Pi}_{Q}(\div \phi_u^h))+b^{*}(u,u,\phi_u^h)-b^{*}(w^h,w^h,\phi_u^h).
\end{gather*}
We can write the nonlinear terms as follows,
\begin{align*}
b^{*}(u,u,\phi_u^h)-b^{*}(w^h,w^h,\phi_u^h)=b^{*}(\eta_u,  u, \phi_u^h)-b^{*}(\phi_u^h,  u, \phi_u^h)+b^{*}( w^h, \eta_u, \phi_u^h).
\end{align*}
Next, we find the bounds for the terms on the right hand side. For the first two terms on the right, use the Cauchy Schwarz 
and Young's inequality, 
\begin{align*}
    &\bigg(\frac{d}{dt}\eta_u,\phi_u^h\bigg)\leq \|\frac{d}{dt}\eta_u\|_{-1}\|\grad \phi^h_u\|
    \leq \frac{\nu}{2}\|\grad \phi^h_u\|^2+C(\nu)\|\frac{d}{dt}\eta_u\|_{-1}^2.
    \end{align*}
    \begin{align*}
    &\alpha^{-2}\bigg(\frac{d}{dt}(p-\widetilde{p}+2\beta\div \eta_u),\lambda^h-\widetilde{ p}+2\beta {\Pi}_{Q}(\div \phi_u^h)\bigg)
    \\&\leq \frac{\alpha^{-2}}{2}\|\frac{d}{dt}(p-\widetilde{p}+2\beta\div \eta_u)\|^2+\frac{\alpha^{-2}}{2}\|\lambda^h-\widetilde{ p}+2\beta {\Pi}_{Q}(\div \phi_u^h)\|^2.
\end{align*}
\begin{align*}
    \alpha^{-2}(p_t,\lambda^h-\widetilde{ p}+2\beta {\Pi}_{Q}(\div \phi_u^h))\leq \frac{\alpha^{-2}}{2}\|p_t\|^2+\frac{\alpha^{-2}}{2}\|\lambda^h-\widetilde{ p}+2\beta {\Pi}_{Q}(\div \phi_u^h)\|^2.
\end{align*}
Next, for the three nonlinear terms, applying \Cref{trilinear_ineq} and H\"older's inequality, we get the following estimates (see \cite[p.381]{john2016finite}),
\begin{align*}
     &|b^{*}(\eta_u,  u, \phi_u^h)|
    \leq\frac{\nu}{4}\|\grad \phi_u^h\|^2+C(\nu)\|\grad u\|^2\|\grad\eta_u\|^2,
    \\&| b^{*}(\phi_u^h,  u, \phi_u^h)|
    \leq \frac{\nu}{16}\|\grad \phi_u^h\|^2+C(\nu)\|\grad u\|^4\|\phi_u^h\|^2,
    \\&|b^{*}( w^h, \eta_u, \phi_u^h)|
    \leq \frac{\nu}{16}\|\grad \phi_u^h\|^2+C(\nu)\|w^h\|\|\grad w^h\|\|\grad\eta_u\|^2.
\end{align*}
Collecting all the terms, combining similar terms, and multiplying by 2, we have,
\begin{gather*}
    \frac{d}{dt}\left\{\| \phi^h_u\|^2+\alpha^{-2}\| \lambda^h-\widetilde{ p}+2\beta \Pi_{Q}(\div \phi_u^h)\|^2\right\}+\frac{\nu}{4}\|\grad \phi^h_u\|^2 +4\beta \|\Pi_Q(\div  \phi_u^h)\|^2
    \\ 
    \leq \alpha^{-2}\|\frac{d}{dt}(p-\widetilde{p}+2\beta\div \eta_u)\|^2+C(\nu)\|\grad u\|^4\|\phi_u^h\|^2+\frac{1}{2}\|\phi_u^h\|^2+ {\alpha^{-2}}\|p_t\|^2
    \\+C(\nu)\Big[\|\frac{d}{dt}\eta_u\|_{-1}^2\|+\|\grad u\|^2\|\grad\eta_u\|^2+\|w^h\|\|\grad w^h\|\|\grad\eta_u\|^2\Big]
    \\+2\alpha^{-2}\|\lambda^h-\widetilde{ p}+2\beta \Pi_{Q}(\div \phi_u^h)\|^2.
\end{gather*}
Denote $a(t):=\max\{2,\frac{1}{2}+C(\nu)\|\grad u\|^4\}$ and its antiderivative is 
\begin{equation*}
    A(t):=\int_0^t a(t')\ dt'<\infty\ \text{for}\ 
    \grad u\in L^4(0,T;L^2(\Omega)).
\end{equation*}
First, we multiply through by the integrating factor $e^{-A(t)}$. Then, integrating over $[0,T]$ and multiplying through by $e^{A(t)}$ give the following:
\begin{gather*}
    \left\{\| \phi^h_u(T)\|^2+\alpha^{-2}\|(\lambda^h-\widetilde{ p}+2\beta {\Pi}_{Q}(\div \phi_u^h))(T)\|^2\right\}
    \\+\int_0^T\Big( \frac{\nu}{4}\|\grad \phi_u^h\|^2+2\beta\|{\Pi}_{Q}(\div \phi^h_u)\|^2\Big) dt \\
    \leq e^{A(t)}\Big\{\Big(\| \phi^h_u(0)\|^2+\alpha^{-2}\| (\lambda^h-\widetilde{ p}+2\beta {\Pi}_{Q}(\div \phi_u^h))(0)\|^2\Big) 
    \\+\int_0^T\Big[ \alpha^{-2}\|\frac{d}{dt}(p-\widetilde{p}+2\beta\div \eta_u)\|^2+ {\alpha^{-2}}\|p_t\|^2+C(\nu)\Big[\|\frac{d}{dt}\eta_u\|_{-1}^2\|
    \\+\|\grad u\|^2\|\grad\eta_u\|^2
    +\|w^h\|\|\grad w^h\|\|\grad\eta_u\|^2\Big]
    \Big] dt\Big\}.
\end{gather*}
Notice that, 
\begin{align*}
\alpha^{-2}\|\frac{d}{dt}(p-\widetilde{p}+2\beta\div \eta_u)\|^2&=\alpha^{-2}(\|(p-\widetilde{p})_t\|+2\beta \|\div{(u-\widetilde{u})_t}\|)^2,
\\&\leq\alpha^{-2}(2\|(p-\widetilde{p})_t\|^2+4\beta \|\div{(u-\widetilde{u})_t}\|^2),
\\&\leq\alpha^{-2}(2\|(p-\widetilde{p})_t\|^2+4\beta C \|\grad{(u-\widetilde{u})_t}\|^2),
\\&\leq \big(\frac{2\alpha^{-2}}{(\gamma^h)^2}+4C\beta\big)\|\grad{(u-\widetilde{u})_t}\|^2.
\end{align*}
After applying H\"{o}lder's inequality, 
we get the following:
\begin{gather*}
\int_0^T \|\grad u\|^2\|\grad\eta_u\|^2dt\leq \|\grad u\|^2_{L^4(0,T;L^2(\Omega))}\|\grad\eta_u\|^2_{L^4(0,T;L^2(\Omega))},
\\
\int_0^T \|w^h\|\|\grad w^h\|\|\grad \eta_u\|^2\leq \|w^h\|_{L^{\infty}(0,T;L^2(\Omega))}\|\grad w^h\|_{L^2(0,T;L^2(\Omega))}\|\grad \eta_u\|_{L^4(0,T;L^2(\Omega))}^2.
\end{gather*}
$\|w^h\|_{L^{\infty}(0,T;L^2(\Omega))}$ and $\|\grad w^h\|_{L^2(0,T;L^2(\Omega))}$ are bounded by problem data due to stability bound. 
Using \Cref{errstokes} and \eqref{errst1}, we get, 
\begin{gather*}
    \left\{\| \phi^h_u(T)\|^2+\alpha^{-2}\|(\lambda^h-\widetilde{ p}+2\beta {\Pi}_{Q}(\div \phi_u^h))(T)\|^2\right\}
    \\+\int_0^T\Big( \frac{\nu}{4}\|\grad \phi_u^h\|^2+2\beta\|{\Pi}_{Q}(\div \phi^h_u)\|^2\Big) dt \\
    \leq e^{A(t)}\Big\{\Big(\| \phi^h_u(0)\|^2+\alpha^{-2}\| (\lambda^h-\widetilde{ p}+2\beta {\Pi}_{Q}(\div \phi_u^h)))(0)\|^2\Big)
    +{\alpha^{-2}}\|p_t\|_{L^2(0,T;L^2)}^2
    \\
    +C(\nu,\Omega)\Big[\inf_{v^h\in X^h}\|\grad(u-v^h)\|_{L^4(0,T;L^2)}^2+\inf_{q^h\in Q^h}\|p-q^h\|_{L^4(0,T;L^2)}^2\Big]
    \\
    +C(\nu,\Omega)(1+h^2)\bigg(\inf_{v^h\in X^h}\|\grad(u-v^h)_t\|_{L^2(0,T;L^2)}^2+\inf_{q^h\in Q^h}\|(p-q^h)_t\|_{L^2(0,T;L^2)}^2\bigg)
    \Big\}.
\end{gather*}
Using the approximation properties \cref{prop} of the spaces $(X^h,Q^h)$, we get,
\begin{gather*}
    \left\{\| \phi^h_u(T)\|^2+\alpha^{-2}\|(\lambda^h-\widetilde{ p}+2\beta {\Pi}_{Q}(\div \phi_u^h))(T)\|^2\right\}
    \\+\int_0^T\Big( \frac{\nu}{4}\|\grad \phi_u^h\|^2+2\beta\|{\Pi}_{Q}(\div \phi^h_u)\|^2\Big) dt \\
    \leq e^{A(t)}\Big\{\Big(\| \phi^h_u(0)\|^2+\alpha^{-2}\| (\lambda^h-\widetilde{ p}+2\beta {\Pi}_{Q}(\div \phi_u^h))(0)\|^2\Big)+{\alpha^{-2}}\|p_t\|_{L^2(0,T;L^2(\Omega))}^2
    \\+C(\nu,\Omega)\Big[h^{2m}(\|u\|^2_{L^4(0,T;H^{m+1}(\Omega))}+\|p\|^2_{L^4(0,T;H^{m}(\Omega))})\Big]
    \\+\Big[(C(\nu,\Omega)h^{2m}
    +C(\nu,\Omega)h^{2+2m})(\|u_t\|^2_{L^2(0,T;H^{m+1}(\Omega))}+\|p_t\|^2_{L^2(0,T;H^{m}(\Omega))})\Big]
    \Big\}.
\end{gather*}
Dropping the pressure term from $\|\lambda^h-\widetilde{ p}+2\beta {\Pi}_{Q}(\div \phi_u^h)\|$ and using the triangle inequality: $\|{e_u}\|\leq \| \phi_u^h\|+\|{\eta_u}\|$, we obtain the desired error estimate.
\end{proof}
\begin{remark}
For the Taylor-Hood finite element pair, we have the following error estimate:
\begin{gather*}
    \sup_{0\leq t\leq T}(\|(u-w^h)(t)\|^2+2\alpha^{-2}\|2\beta{\Pi}_{Q}(\div{(u-w^h))(t)}\|^2)
    \\+\int_0^T\Big( \frac{\nu}{4}\|\grad (u-w^h)\|^2+2\beta\|{\Pi}_{Q}(\div (u-w^h))\|^2\Big) dt 
    \leq \mathcal{O}(h^4).
\end{gather*}
\end{remark}
\section{Numerical Tests}\label{sec:6}
In this section, we conduct several numerical tests to evaluate the performance of the scheme we consider and compare it with PP and AC methods. Stability, convergence, damping of oscillations, and qualitative character of the method are evaluated. Throughout all computations, we consider the inf-sup stable Taylor-Hood finite element pair and use the publicly licensed finite element software FreeFEM++ \cite{hec12}.
\subsection{Numerical Stability} As a first numerical experiment, we depict the behavior of problem variables in a well-known test so-called the Taylor-Green vortex. In this setting, the initial conditions are taken as:
\begin{align*}
    & u(x,y,t)=e^{-2t / Re}(\cos x \sin y ,-\cos y \sin x),\\
    &p(x,y,t)=-\frac{1}{4}e^{-4t / Re}(\cos 2x + \cos 2y) .
\end{align*}
The test is applied on a unit square domain with a mesh resolution of $16 \times 16$ and velocity boundary conditions are set the same as the initial conditions all over the computational domain. We depict $\| w\|$, $\| \div w\|$, and $\| \lambda\|$ to expose the solution behavior in the time interval $[0,10]$. The Reynolds number is picked as $Re=1$. The results under the selection of different kinds of parameters are seen in \Cref{fig:stab}. As expected, increasing the parameters $\alpha$ and $\beta$ by taking them as reciprocal of $\Dt t$ results in smaller norm values. In this test setup, the time step is chosen as $\Dt t=0.1$ and we would like to point out that the smaller $\Dt t$ values, which means bigger $\alpha$ and $\beta$, would yield the smaller norms in the case of $\alpha^2=O(\Dt t ^{-1}),\,\,\beta=O(\Dt t ^{-1})$. We also test the case of small $\alpha$ and $\beta$ to better understand the effect of these parameters. The results obtained from this test verify the stability results which were proven theoretically in \Cref{sec:4}. 
\begin{figure}[h!]
\subfloat[$\alpha^2=O(\Dt t),\,\,\beta=O(\Dt t)$]{\includegraphics[width=6.5cm]{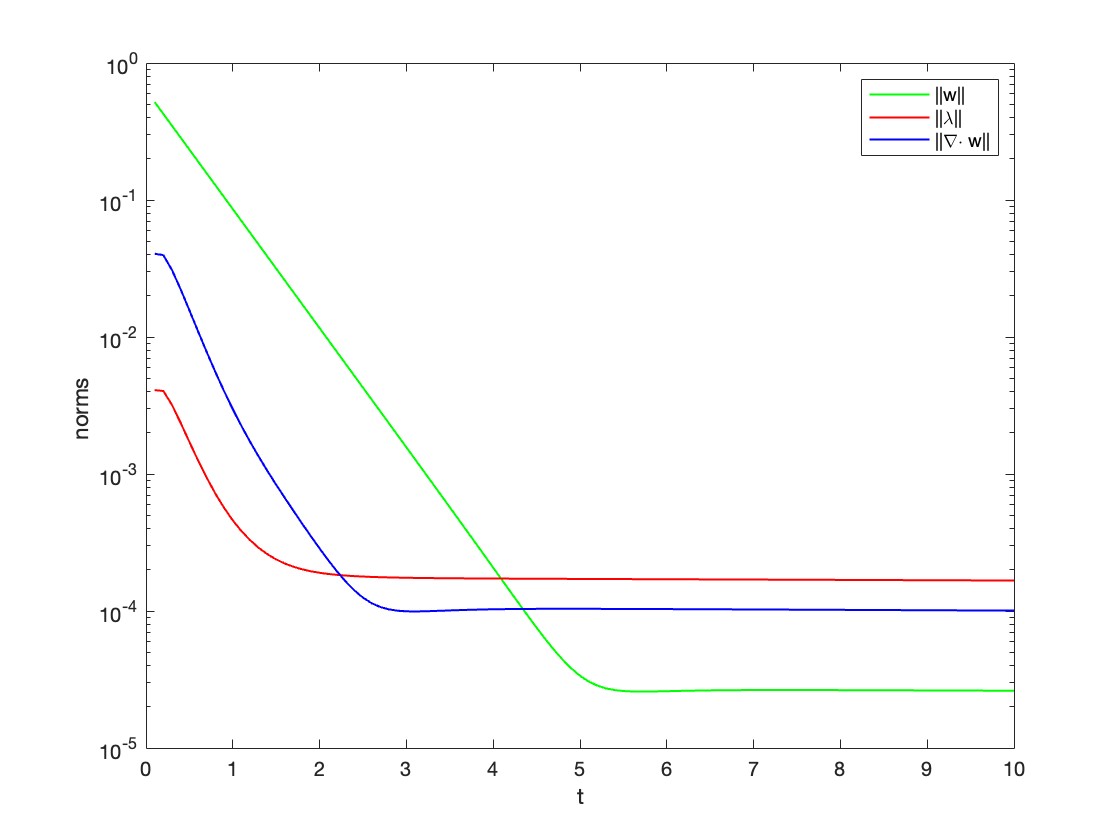}}
\subfloat[$\alpha^2=O(\Dt t ^{-1}),\,\,\beta=O(\Dt t ^{-1})$]{\includegraphics[width=6.5cm]{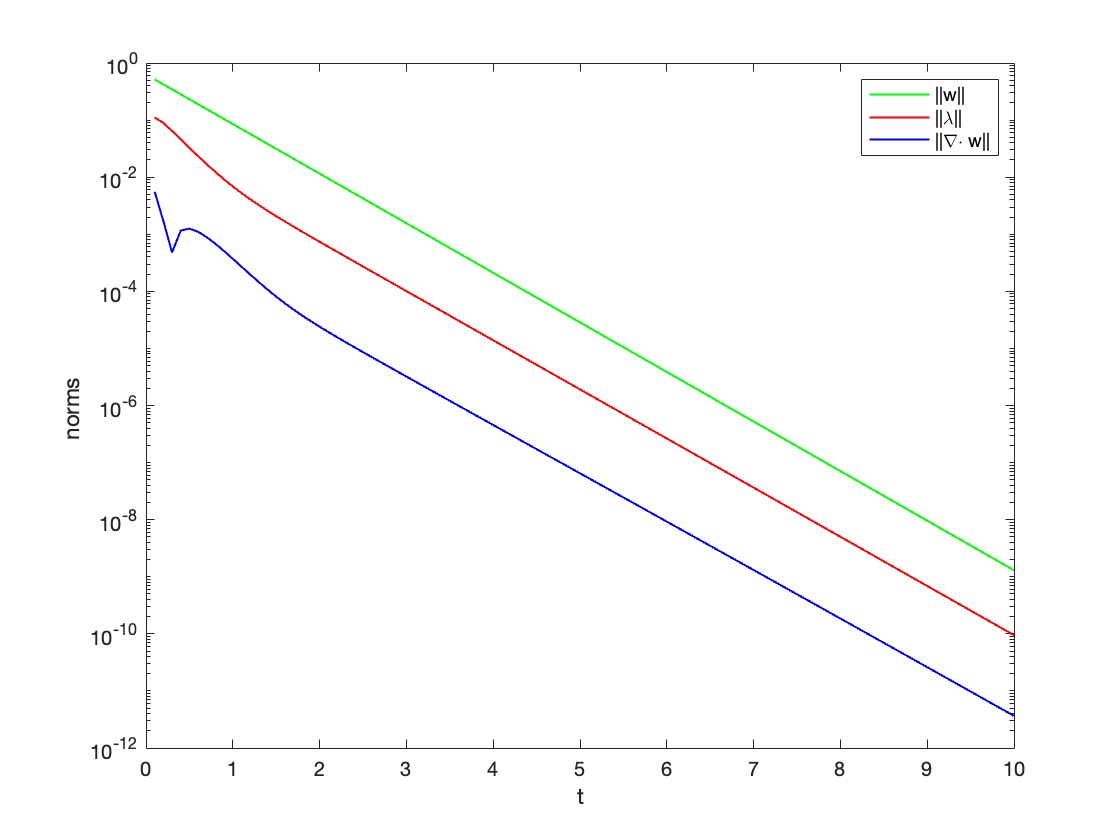}}
\caption{Velocity and pressure norms for different selections of $\alpha$ and $\beta$.}
\label{fig:stab}
\end{figure}
\subsection{Test of Accuracy} In this numerical example, we verify the temporal convergence results obtained theoretically in \Cref{sec:5}. Since a continuous in time analysis was covered therein, we consider a BE time discretization here which is described in \eqref{nse_fem_be13} and \eqref{nse_fem_be23}. The following exact solution for the NSE is considered in the domain $\Omega=(0,1)\times(0,1)$. The velocity and the pressure are taken as:
\begin{align*}
    & u(x,y,t)=e^t(\cos y,\sin x),\\
    &p(x,y,t)=(x-y)(1+t).
\end{align*}
The body force $f(x,t)$ is found by plugging these in the NSE. The other parameters are picked as $\alpha^2=O(\Dt t ^{-1}),\,\,\beta=O(\Dt t ^{-1})$, and $Re=1$. We choose a fine mesh resolution of $ 128 \times 128$ to isolate the spatial contribution of the errors despite they are negligibly small and run the code in time interval $[0,1]$. \par
The measurements were made by norm $L^2([0,T],\Omega)$ which is defined e.g. for $w$ as follows:
$$\| w\|:=\left(\int_0^T\| w(\cdot, t)\|_{L^2(\Og)}^2\ dt\right)^{1/2}.$$
\begin{table}[H]
    \centering
    \begin{tabular}{||c|c|c|c|c|c|c||}
    \hline
         $\Dt t$ & $\|u-w^h\|$ &rate& $\|p-\lambda^h\|$ &rate& $\|\div{w}\|$  \\
         \hline
         0.5&0.00162&-&0.5248&-&1.75e-8\\
         \hline
         0.25&0.00077&1.07&0.0422&1.10&1.06e-8
 \\
         \hline
         0.125&0.00038&1.01&0.0219&1.06&1.09e-8
 \\
         \hline
         0.0625&0.00018&1.07&0.0118&1.06&1.2e-8 
\\
         \hline
         0.03125&9.01e-5&0.99&0.0058&1.02&1.4e-8
\\
         \hline
    \end{tabular}
    \caption{Errors, rates of convergence, and values of $\|\div{w}\|$.}
    \label{tab:accu}
\end{table}
In \Cref{tab:accu}, we observe first order convergence for both velocity and pressure. This is an optimal rate, thanks to the BE time discretization. We also present the values of $\| \div{w}\|$ to have an idea of how small the divergence values are. According to the results of this accuracy test, theoretical expectations are compatible with numerical results. Higher Reynolds numbers and higher order temporal discretizations were also tested, not reported herein and expected rates are obtained. We only keep the current case for the sake of brevity. 
\subsection{Damping of Oscillations}\label{sec:6.2} 
The acoustic equations for the model are based on assuming that $u(x,t)$ is small while the pressure is not. Thus, we start by taking the divergence of the momentum equation and first order time derivative of the continuity equation in \eqref{nse} and then eliminate $\div{w_t}$ to get:
$$\lambda_{tt}-2\beta \Delta \lambda_t-\alpha^2\Delta \lambda=-\div{f}+\div{(w\cdot\grad w+\frac{1}{2}(\div w)w)}.$$
This is the model's Lighthill pressure wave equation. To analyze nonphysical acoustics, we set $\div{f}=0$ and drop the quadratic term $\div{(w\cdot\grad w+\frac{1}{2}(\div w)w)}$. This setting gives us the acoustic equation:
$$\lambda_{tt}-2\beta \Delta \lambda_t-\alpha^2\Delta \lambda=0 .$$  
This is over-damped (hence no oscillations) provided:
\begin{eqnarray*}
\frac{\alpha}{\beta} < \sqrt{\sigma_{min}(-\Delta^h)}
\end{eqnarray*}
 with $\sigma_{min}(-\Delta^h)$ being the smallest eigenvalue of discrete Laplacian. Thus, a selection of $\alpha^2=O(\Dt t ^{-1}),\,\,\beta=O(\Dt t ^{-1})$ which results with $\frac{\alpha}{\beta}=O(\Dt t ^{1/2}) < \sqrt{\sigma_{min}(-\Delta^h)}$ suggests no pressure oscillations in this heuristic analysis.\par
To test the performance of the hybrid algorithm in this paper, we conduct this numerical test and compare the results of the considered scheme with AC and PP methods in terms of reducing the nonphysical oscillations.\par
The domain is a disk that includes another disk as an obstacle inside with a smaller radius and off the center. The diameters and centers are defined by $r_1=1,r_2=0.1,c=(c_1,c_2)=(1/2,0)$. Hence, the overall computational domain becomes
\beas
\Omega=\{(x,y):x^2+y^2< r_1^2 \;\text{and}\; (x-c_1)^2+(y-c_2)^2> r_2^2\}.
\eeas
The computational domain with its triangulation is presented in \Cref{fig:offset}. A counterclockwise rotation drives the flow with the body force given below:
\beas
f(x,y,t)=(-4y(1-x^2-y^2),4x(1-x^2-y^2))^T.
\eeas
\begin{figure}[h!]
    \centering
    \includegraphics[width=0.75\linewidth]{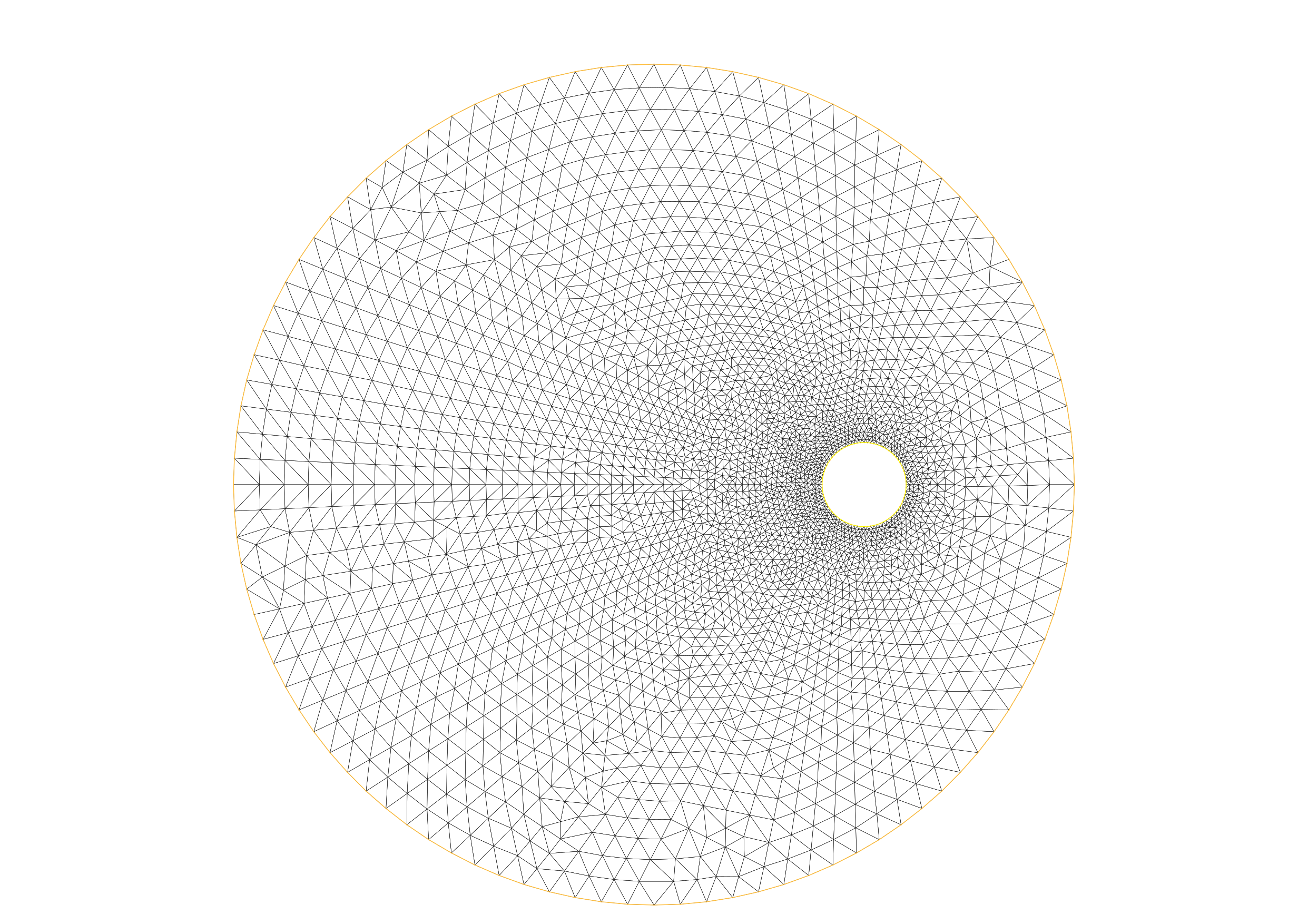}
    \caption{Triangulated computational domain.}
    \label{fig:offset}
\end{figure}
Boundary conditions are taken as no-slip for both circles. $100$ mesh points around the outer circle and $80$ mesh points around the inner circle are taken as mesh resolution. A counterclockwise force propels the flow ($f = 0$ on the outside circle). Thus, the flow revolves around the origin and interacts with the submerged circle. Initial conditions are taken to be the solution of a steady state Stokes problem and we pick $\alpha^2=O(\Dt t ^{-1}),\,\,\beta=O(\Dt t ^{-1})$ for the first order BE method. We also test a BE+time filter method from \cite{guzel22} by adding a time filter. For the second order trapezoidal scheme, this selection alters as $\alpha^2=O(\Dt t ^{-2}),\,\,\beta=O(\Dt t ^{-2})$ to preserve the method's time accuracy. We have $Re=1000$ for this test case. The code is run for 3 different cases of time discretization and details are given for each separately. In each case, the test is carried out in time interval $[0,20]$ with $\Dt t=0.01$.
\subsubsection{Backward Euler}\label{BE}
We first directly implement the scheme with \eqref{nse_fem_be13} and \eqref{nse_fem_be23}, which considers the Backward Euler time discretization (BE). For measuring the oscillations, we calculate the discrete curvature
$$\kappa_{new} = \|\lambda_{new}^{n+1}-2\lambda_{new}^{n}+\lambda_{new}^{n-1}\|$$
where n denotes the time step and $\lambda_{new}=\lambda+2\beta \div{w}$ is a modified pressure. This kind of pressure-like variable was inspired by \cite{linke2011}. Similar results were obtained by testing the pressure discrete curvature. We also give the evolution of $\|\div{w}\|$ to have more clear ideas on the performance of each scheme.
\begin{figure}[h!]
\subfloat[Discrete pressure oscillations]{\includegraphics[width=6.5cm]{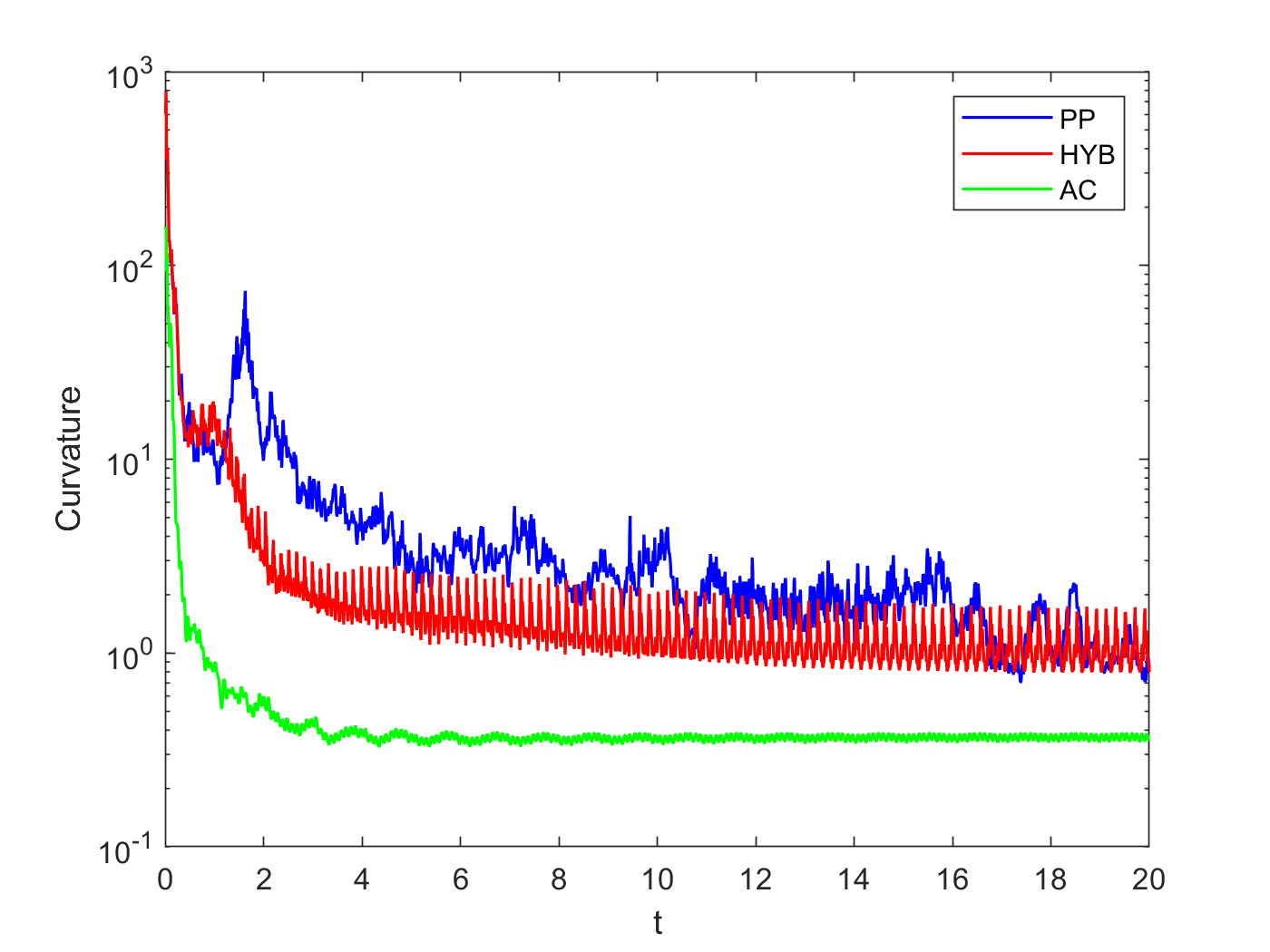}}
\subfloat[Evolution of divergence norm]{\includegraphics[width=6.5cm]{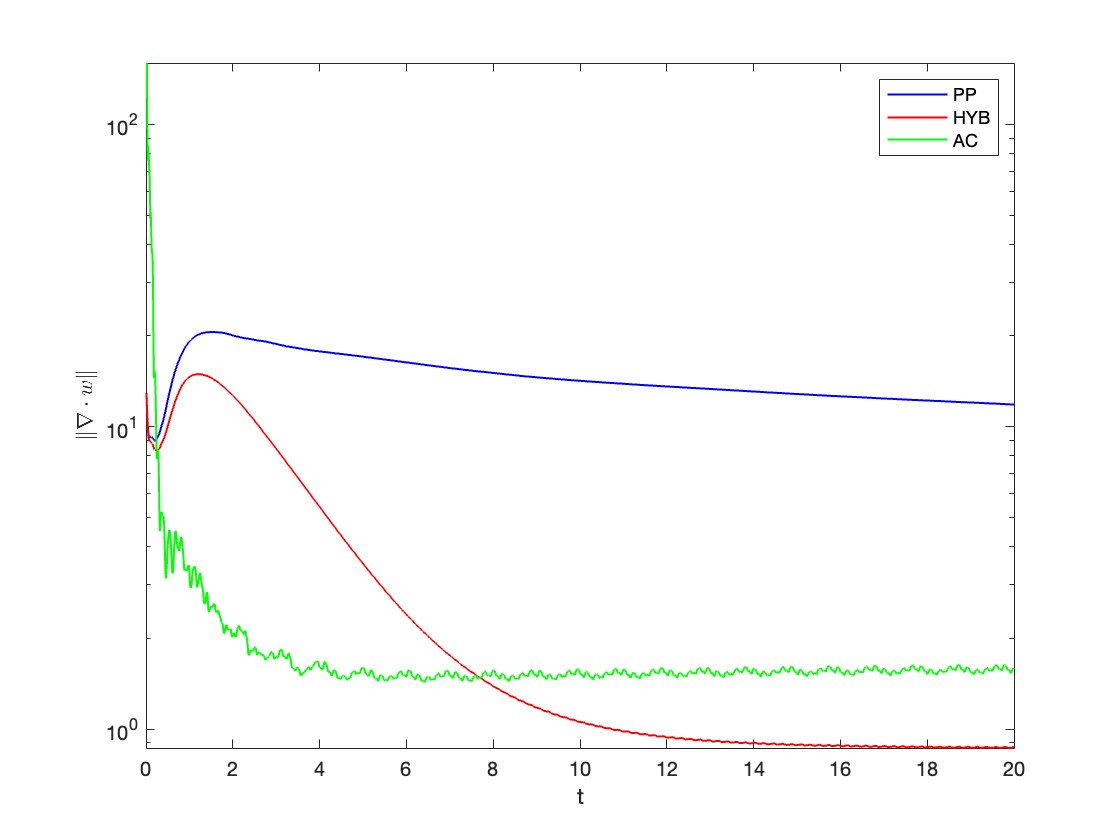}}
\caption{Offset circles test results for BE time discretization.}
\label{fig:BE}
\end{figure}
As one inspects the curvature diagrams in \Cref{fig:BE}, it is clear that the performance of each three schemes in terms of damping non-physical acoustic waves is quite similar. Particularly, the hybrid scheme and the AC method have almost the same effect on pressure oscillations with only a very slight difference in magnitude. The PP method has produced slightly larger  discrete curvature values when compared to others. The divergence norms are also compatible with the oscillation case. The hybrid scheme and the AC method have smaller divergence values when compared with the PP method. Also, after $t=10$ the hybrid scheme displays slightly smaller values than the AC.\par
Since the BE time discretization has a large amount of numerical dissipation, these results are plausibly due to the damping effect of temporal discretization being greater than model influences. We would also notice here that the results obtained here for the AC scheme are completely compatible with the results given in \cite{guzel22}.
\subsubsection{Backward Euler + time filter}\label{BETF}
In this part, we post-process the solution obtained from \eqref{nse_fem_be13} and \eqref{nse_fem_be23} with an active time filter related to the Robert-Asselin filter, Robert \cite{robert69}, Asselin \cite{asse72}. These filters have also been used to suppress nonphysical oscillations in numerical methods \cite{decaria2019}. In each time step, one should only update the solution with
\beas
w^{n+1}\leftarrow w^{n+1}-\frac{\mu}{2}(w^{n+1}-2w^{n}+w^{n-1}),\\
\lambda^{n+1}\leftarrow \lambda^{n+1}-\frac{\mu}{2}(\lambda^{n+1}-2\lambda^{n}+\lambda^{n-1}),
\eeas
where $\mu$ typically is $\mu=0.1$ as suggested in \cite{guzel22}. We verified from our numerical tests that this selection of $\mu$ is better when compared to $\frac{\mu}{2}=1/3$. Thus, the accuracy is boosted by a minimal computational effort. Although \eqref{nse_fem_be13} and \eqref{nse_fem_be23} itself is highly dissipative and dampens the oscillations, we would like to observe the effect of these time filters on different methods.
\begin{figure}[h!]
\subfloat[Discrete pressure oscillations]{\includegraphics[width=6.5cm]{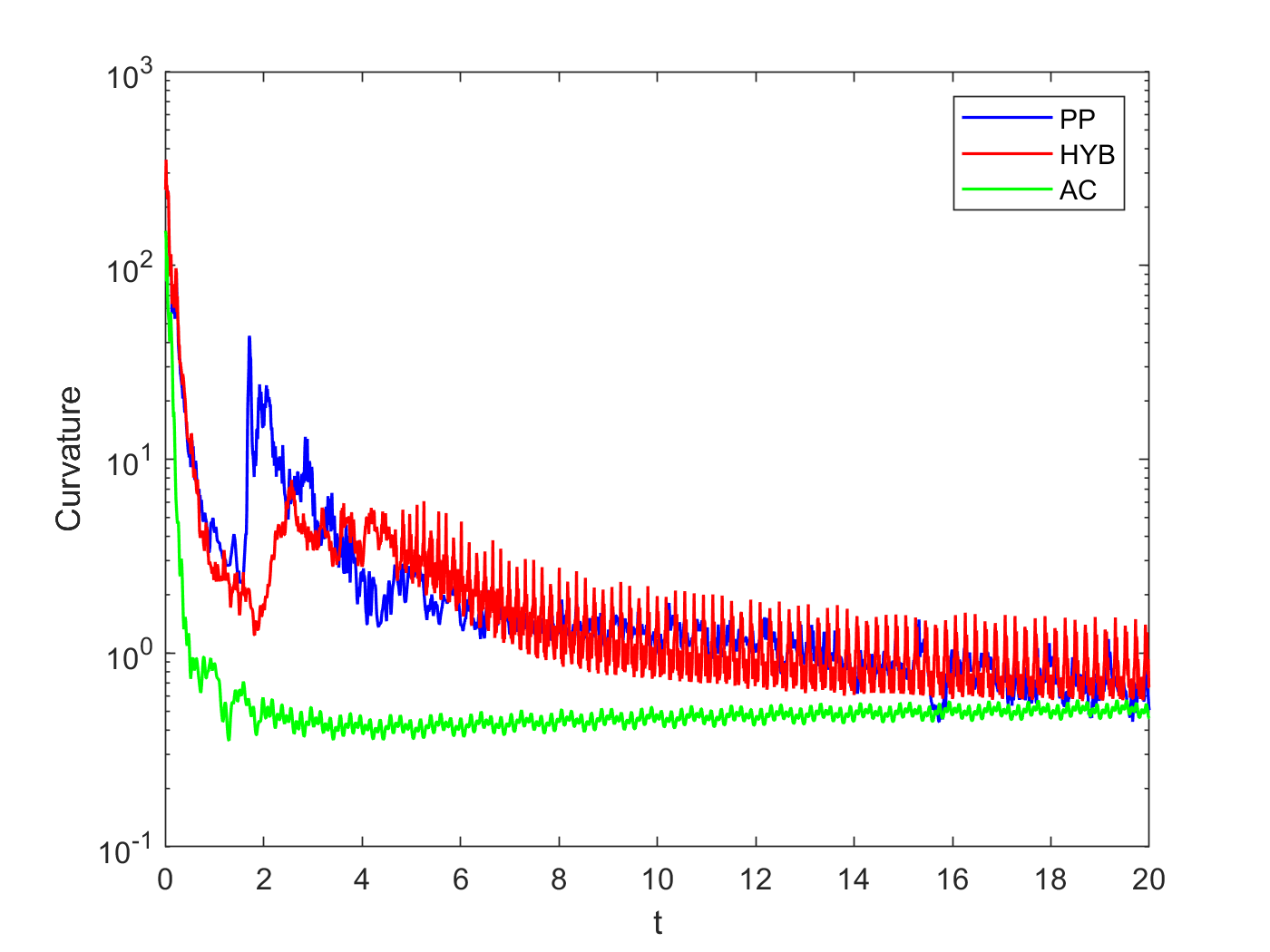}}
\subfloat[Evolution of divergence norm]{\includegraphics[width=6.5cm]{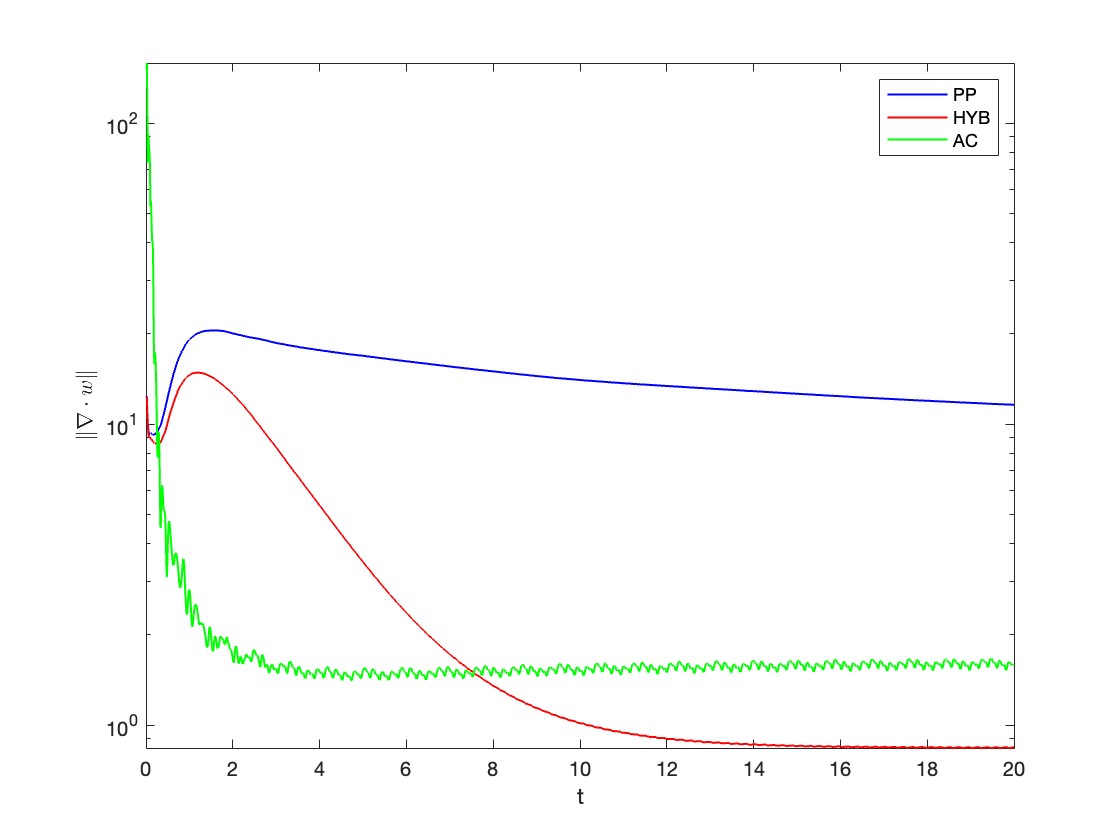}}
\caption{Offset circles test results for BE+time filter.}
\label{fig:BETF}
\end{figure}
We can deduce from \Cref{fig:BETF} that applying the time filters has a minimal effect in terms of damping the oscillations due to numerical dissipation. Furthermore, in terms of divergence, the hybrid method seems to have smaller divergence norms upon the application of time filters. When compared with the BE case, time-filtered solutions have almost half-order smaller divergence norms than the AC method, thanks to the time filters.
\subsubsection{Trapezoidal}\label{sec:CN}
To better understand and compare the oscillation suppressing performance of the method, we conduct the test with a conservative trapezoidal scheme. Thus, we can see the effect of each method on oscillations without damping effects of the numerical dissipation. We modify the hybrid method's algorithm as a second order trapezoidal scheme and interpret the new discretization as follows: Denote  $w^* =\frac{3}{2}w^h_{n}-\frac{1}{2}w^h_{n-1}$ and $w_{n+1/2}^h=\frac{1}{2}w^h_{n}+\frac{1}{2}w^h_{n+1}$. For all $v^h\in X^h,\,\, q^h \in Q^h $
\begin{equation}
\begin{aligned}
\label{algocn}
 &\Big( \frac{ w_{n+1}^h- w_n^h}{k}, v^h\Big)+b^*( w^*, w^h_{n+1/2}, v^h)  +\nu(\grad  w^h_{n+1/2},\grad  v^h)
 + (\grad \lambda^h_{n+1/2}, v^h)
   \\& =(f_{n+1/2}( x), v^h), 
    \\
    &\Big( \frac{ \lambda_{n+1}^h- \lambda_n^h}{k}, q^h\Big)=-\alpha^2 ( \div w^h_{n+1/2},q^h )
    -2\beta \Big( \div (\frac{ w^h_{n+1}-  w^h_{n}}{k}), q^h\Big).
\end{aligned}
\end{equation}
One can decouple \eqref{algocn} as described for \eqref{nse_fem_be13} and \eqref{nse_fem_be23}. The results are presented in \Cref{fig:cn}.
\begin{figure}[H]
\subfloat[Discrete pressure oscillations]{\includegraphics[width=6.5cm]{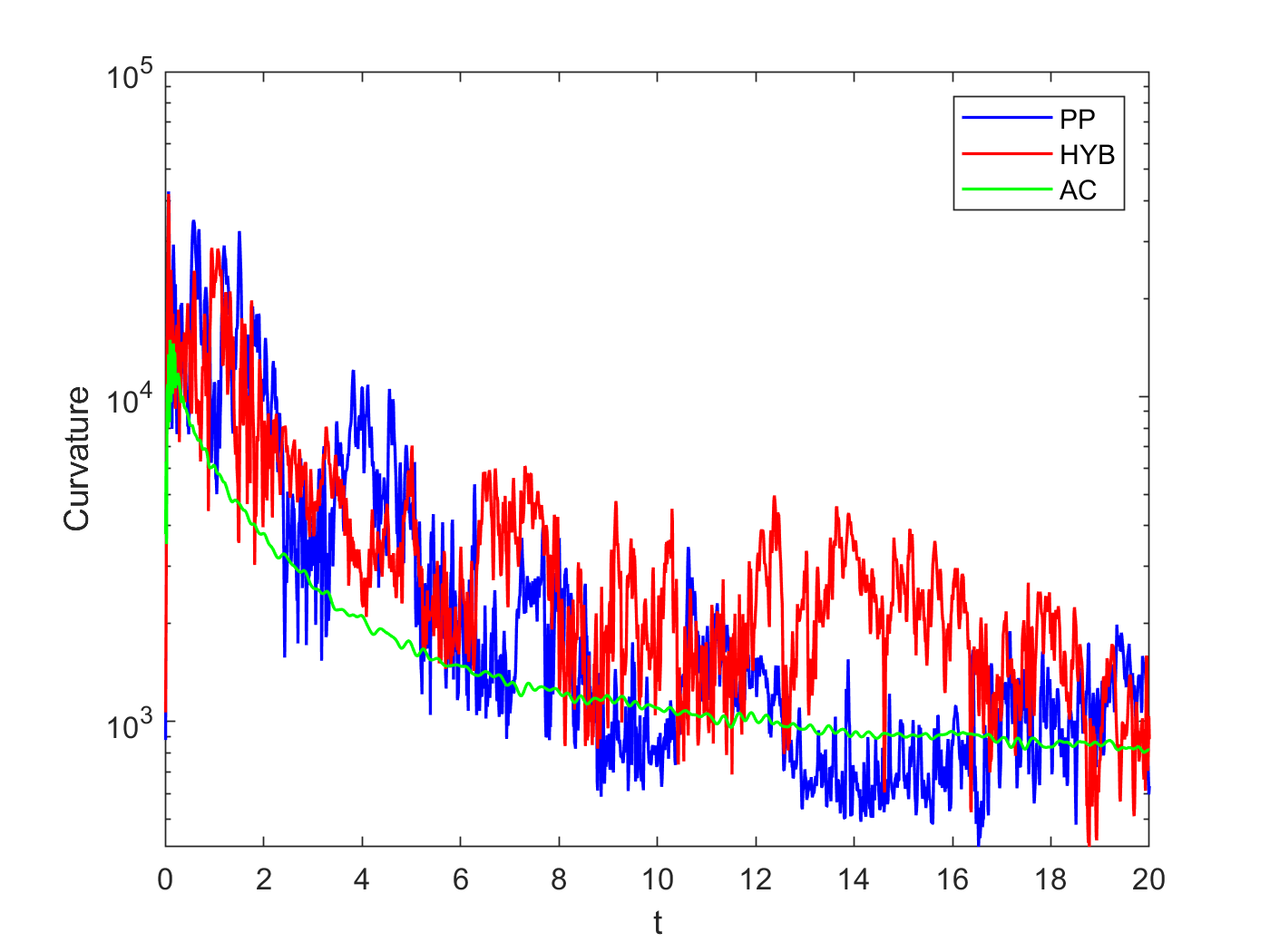}}
\subfloat[Evolution of divergence norm]{\includegraphics[width=6.5cm]{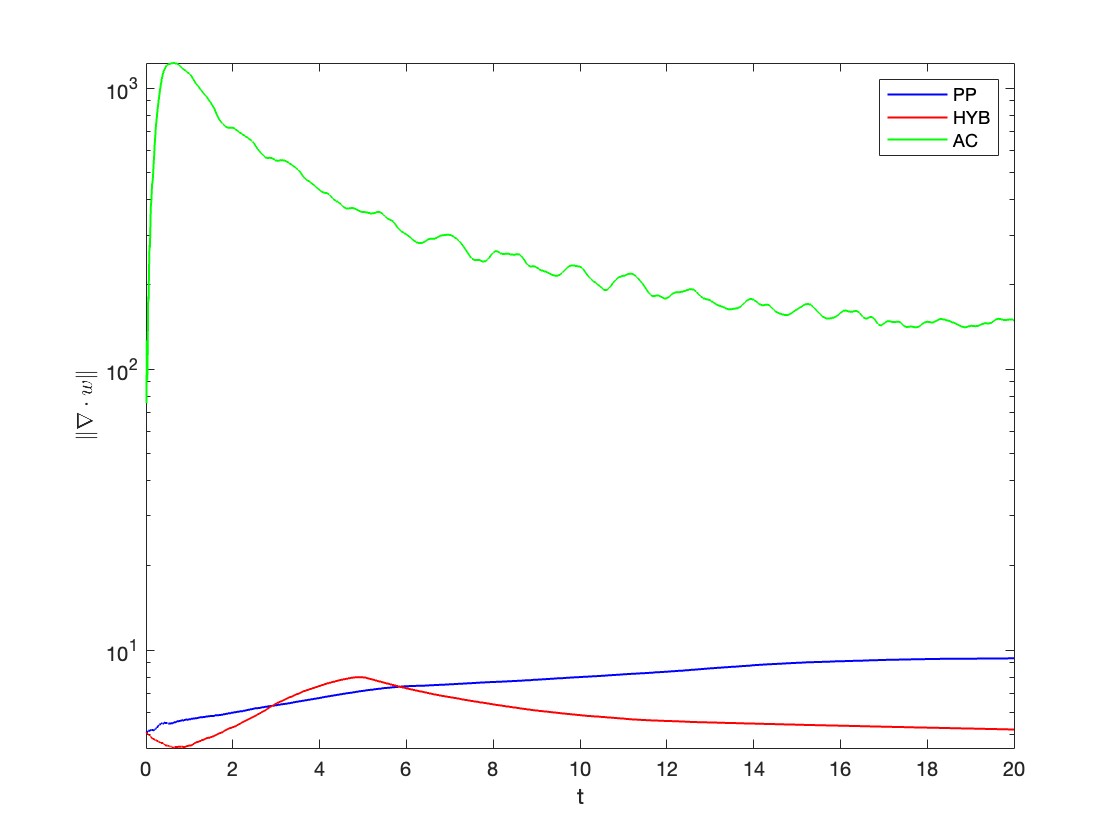}}
\caption{Offset circles test results for trapezoidal discretization.}
\label{fig:cn}
\end{figure}
We observe that utilizing a conservative time discretization results in greater levels of nonphysical acoustics. All three methods perform similarly with a slightly better performance of the penalty projection scheme. In means of divergence norms, the hybrid method and the penalty scheme are almost identical and about two orders smaller than the AC method. 
\subsection{Flow over a full step inside a channel}
In this last numerical example, we consider a benchmark flow over forward-backward facing step which is given and described in \cite{gunz1989}. The domain is a $40 \times 10 $ rectangular channel with a $1 \times 1 $ step into the channel at the lower wall. No slip boundary conditions are applied at the horizontal walls and a parabolic inflow and outflow are applied vertically. The initial velocity, inflow, and outflow profiles are given by:
$$u=(y(10-y)/25,0).$$
The other problem parameters are taken as $Re=600$, $\alpha^2=O(\Dt t ^{-2}),\,\,\beta=O(\Dt t ^{-2})$. We use \eqref{algocn} for this test problem. We simulate for the time interval $[0,40]$ with $\Dt t= 0.01$. The computational domain with triangulation is given in \Cref{fig:chan_dom}.
\begin{figure}[h!]
    \centering
    \includegraphics[width=1.0\linewidth]{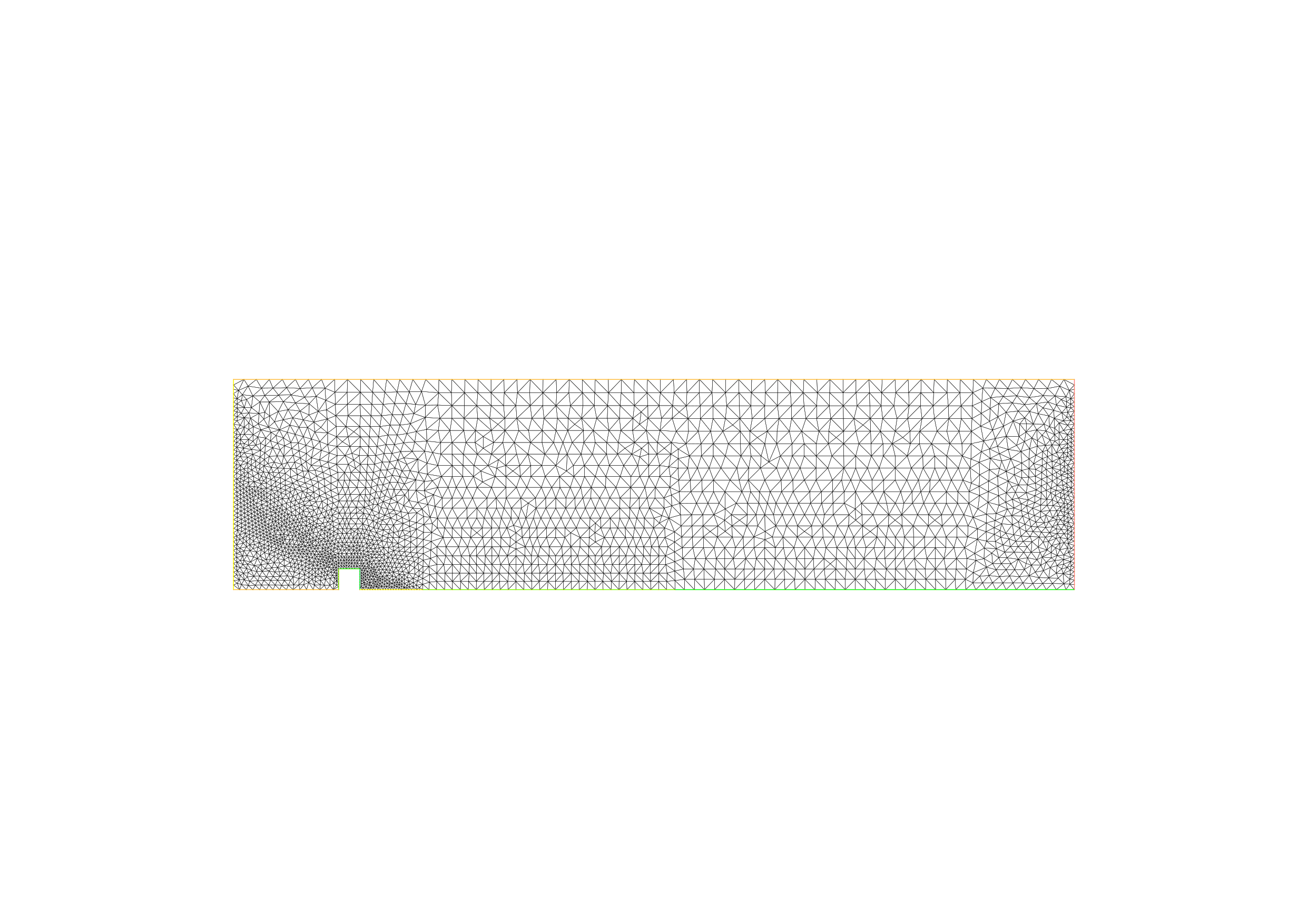}
    \caption{Computational domain and triangulation for the channel flow.}
    \label{fig:chan_dom}
\end{figure}
A total of $ 26931 $ degrees of freedom is provided from the mesh selected here. After the simulation is done, it is expected that the initial parabolic flow profile changes from parabolic to a plug-like behavior, and the shedding of eddies behind the step is decreased.
\begin{figure}[H]
    \centering
    \includegraphics[width=0.8\linewidth]{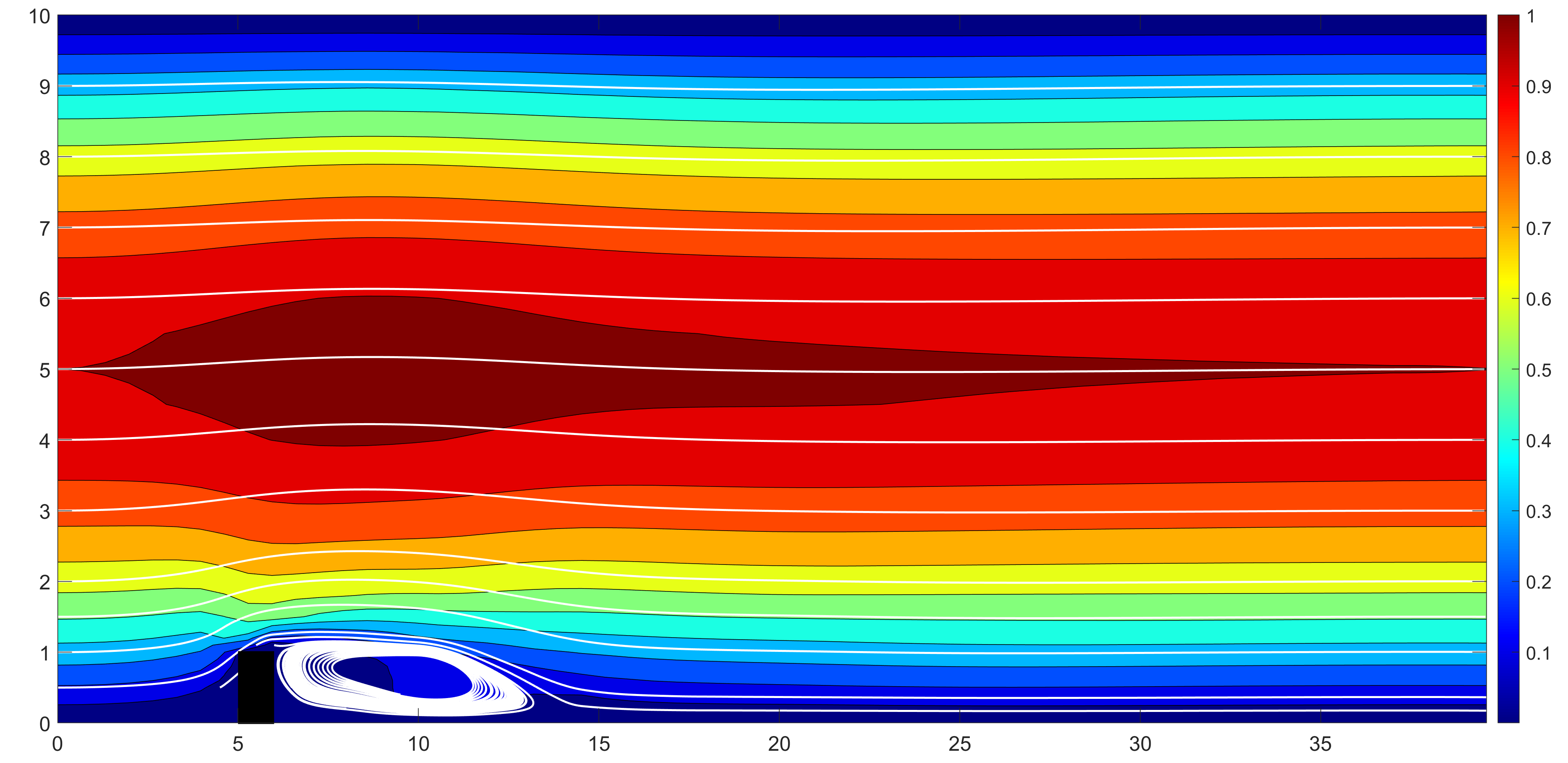}
    \caption{Streamlines over speed contours for the channel flow.}
    \label{fig:chan}
\end{figure}
We can easily observe from \Cref{fig:chan} that, the parabolic inflow has been altered and a large non-separated  eddy has been formed behind the step. These key observations suggest that the scheme considered here has captured the correct flow behavior and gives qualitatively correct results.
\section{Conclusion and future prospects}
A complete finite element numerical analysis is given and supported with extensive numerical examples for the hybrid method. The comparisons suggest that the hybrid scheme performance fits between the pressure penalty and artificial compression schemes. One key point is that the experiments' behavior in damping pressure oscillations is different than predicted by a heuristic analysis of the acoustic equations (This difference is obviously due to the heuristic acoustic analysis neglecting non-negligible effects). The results on damping $\|\div{w}\|$ also differ here (for implicit methods) from those in \cite{ramshaw1991hybrid} (for explicit methods). Explanation of this behavior is an open problem.\par
In the next step, we can consider to state and compare another kind of hybridization of divergence penalizing schemes. 
\section*{Acknowledgement}
 We thank the National Science Foundation (NSF) and The Scientific and Technological Research Council of T\"{u}rkiye (TUBITAK) for their support.
\bibliographystyle{siamplain}
\bibliography{references}

\begin{thebibliography}{10}

\bibitem{adams2003sobolev}
{\sc R.~A. Adams and J.~J. Fournier}, {\em {S}obolev spaces}, Elsevier, 2003.

\bibitem{asse72}
{\sc R.~A. Asselin}, {\em Frequency filter for time integrations}, Mon. Weather Rev., 100 (1972), pp.~487--490.

\bibitem{baker1982higher}
{\sc G.~A. Baker, V.~A. Dougalis, and O.~A. Karakashian}, {\em On a higher order accurate fully discrete {G}alerkin approximation to the {N}avier-{S}tokes equations}, Mathematics of Computation, 39 (1982), pp.~339--375.

\bibitem{brooks1982streamline}
{\sc A.~N. Brooks and T.~J.~R. Hughes}, {\em Streamline upwind/{P}etrov-{G}alerkin formulations for convection dominated flows with particular emphasis on the incompressible {N}avier-{S}tokes equations}, Computer methods in applied mechanics and engineering, 32 (1982), pp.~199--259.

\bibitem{ccibik2024convergence}
{\sc A.~{\c{C}}ibik and W.~Layton}, {\em Convergence of a {R}amshaw-{M}esina iteration}, arXiv preprint arXiv:2403.04702,  (2024).

\bibitem{decaria2019}
{\sc V.~DeCaria, W.~Layton, and M.~McLaughlin}, {\em An analysis of the {Robert–Asselin} time filter for the correction of nonphysical acoustics in an artificial compression method}, Numerical Methods for {P}artial {D}ifferential {E}quations, 35 (2019), pp.~916--935.

\bibitem{duk93}
{\sc J.~K. Dukowicz}, {\em Computational efficiency of the hybrid penalty-pseudocompressibility method for incompressible flow}, Computers \& Fluids, 23 (1994), pp.~479--486.

\bibitem{fang2023penalty}
{\sc R.~Fang}, {\em Penalty ensembles for {N}avier-{S}tokes with random initial conditions and forcing}, arXiv preprint arXiv:2309.12870,  (2023).

\bibitem{gunz1989}
{\sc M.~Gunzburger}, {\em Finite Element Methods for viscous incompressible flows: A guide to theory, practice, and algorithms}, Computer science and scientific computing, Academic Press, 1989.

\bibitem{guzel22}
{\sc A.~Guzel, W.~Layton, M.~McLaughlin, and Y.~Rong}, {\em Time filters and spurious acoustics in artificial compression methods}, Numerical Methods for Partial Differential Equations, 38 (2022), pp.~1908--1928.

\bibitem{hec12}
{\sc F.~Hecht}, {\em New development in {FreeFEM++}}, Journal of numerical mathematics, 20 (2012), pp.~251--266.

\bibitem{heinrich1995penalty}
{\sc J.~C. Heinrich}, {\em The penalty method for the {N}avier-{S}tokes equations}, Archives of Computational Methods in Engineering, 2 (1995), pp.~51--65.

\bibitem{john2016finite}
{\sc V.~John}, {\em Finite element methods for incompressible flow problems}, vol.~51, Springer, 2016.

\bibitem{kean2023doubly}
{\sc K.~Kean, X.~Xie, and S.~Xu}, {\em A doubly adaptive penalty method for the {N}avier-{S}tokes equations}, International Journal of Numerical Analysis \& Modeling, 20 (2023).

\bibitem{kobel2002symmetric}
{\sc G.~M. Kobel'kov}, {\em Symmetric approximations of the {N}avier-{S}tokes equations}, Sbornik: Mathematics, 193 (2002), p.~1027.

\bibitem{labovsky2009stabilized}
{\sc A.~Labovsky, W.~J. Layton, C.~C. Manica, M.~Neda, and L.~G. Rebholz}, {\em The stabilized extrapolated trapezoidal finite-element method for the {N}avier-{S}tokes equations}, Computer Methods in Applied Mechanics and Engineering, 198 (2009), pp.~958--974.

\bibitem{layton2020doubly}
{\sc W.~Layton and M.~McLaughlin}, {\em Doubly-adaptive artificial compression methods for incompressible flow}, Journal of Numerical Mathematics, 28 (2020), pp.~175--192.

\bibitem{cfdbook}
{\sc W.~J. Layton}, {\em Introduction to the numerical analysis of incompressible viscous flows}, Society for Industrial and Applied Mathematics, 2008.

\bibitem{linke2017connection}
{\sc A.~Linke, M.~Neilan, L.~G. Rebholz, and N.~E. Wilson}, {\em A connection between coupled and penalty projection timestepping schemes with {FE} spatial discretization for the {N}avier-{S}tokes equations}, Journal of Numerical Mathematics, 25 (2017), pp.~229--248.

\bibitem{linke2011}
{\sc A.~Linke, L.~G. Rebholz, and N.~Wilson}, {\em On the convergence rate of grad-div stabilized {T}aylor–{H}ood to {S}cott–{V}ogelius solutions for incompressible flow problems}, Journal of Mathematical Analysis and Applications, 381 (2011), pp.~612--626.

\bibitem{mchugh1995damped}
{\sc P.~R. McHugh and J.~D. Ramshaw}, {\em Damped artificial compressibility iteration scheme for implicit calculations of unsteady incompressible flow}, International journal for numerical methods in fluids, 21 (1995), pp.~141--153.

\bibitem{olshanskii2019penalty}
{\sc M.~A. Olshanskii and V.~Yushutin}, {\em A penalty finite element method for a fluid system posed on embedded surface}, Journal of Mathematical Fluid Mechanics, 21 (2019), pp.~1--18.

\bibitem{prohl1997projection}
{\sc A.~Prohl}, {\em Projection and quasi-compressibility methods for solving the incompressible {N}avier-{S}tokes equations}, Springer, 1997.

\bibitem{ramshaw1991hybrid}
{\sc J.~D. Ramshaw and G.~L. Mesina}, {\em A hybrid penalty-pseudocompressibility method for transient incompressible fluid flow}, Computers \& fluids, 20 (1991), pp.~165--175.

\bibitem{ramshaw1990accelerated}
{\sc J.~D. Ramshaw and V.~A. Mousseau}, {\em Accelerated artificial compressibility method for steady-state incompressible flow calculations}, Computers \& fluids, 18 (1990), pp.~361--367.

\bibitem{ramshaw1991damped}
{\sc J.~D. Ramshaw and V.~A. Mousseau}, {\em Damped artificial compressibility method for steady-state low-speed flow calculations}, Computers \& fluids, 20 (1991), pp.~177--186.

\bibitem{robert69}
{\sc A.~Robert}, {\em The integration of a spectral model of the atmosphere by the implicit method}, Proc. WMO/IUGG Symp. NWP, Japan Meteorol. Soc., Tokyo, Japan,  (1969), pp.~19--24.

\bibitem{shen1995error}
{\sc J.~Shen}, {\em On error estimates of the penalty method for unsteady {N}avier-{S}tokes equations}, SIAM Journal on Numerical Analysis, 32 (1995), pp.~386--403.

\end{thebibliography}
\end{document}